\documentclass[12pt]{amsart}
\usepackage{amssymb, enumerate}
\theoremstyle{plain}
\newtheorem{theorem}{Theorem}[section]
\newtheorem{prop}[theorem]{Proposition}

\newtheorem{corollary}[theorem]{Corollary}
\newtheorem*{thmA}{Theorem A}
\newtheorem*{thmB}{Theorem B}
\newtheorem*{thmC}{Theorem C}

\theoremstyle{definition}
\newtheorem{definition}[theorem]{Definition}
\newtheorem{remark}[theorem]{Remark}

\newcommand\M{\mathcal{M}}

\allowdisplaybreaks

\begin{document}

\title[Functional calculus on non-doubling manifold with ends]{Functional calculus of operators 
with heat kernel bounds on non-doubling manifolds with ends}

\author{The Anh Bui}

\address{
T. A. Bui, Department of
Mathematics, Macquarie
University, NSW, 2109, Australia}
\email{the.bui@mq.edu.au}

\author{Xuan Thinh Duong}

\address{
X. T. Duong, Department of
Mathematics, Macquarie
University, NSW, 2109, Australia}
\email{xuan.duong@mq.edu.au}

\author{Ji Li}

\address{
J. Li, Department of
Mathematics, Macquarie
University, NSW, 2109, Australia}
\email{ji.li@mq.edu.au}

\author{Brett D. Wick}

\address{
B. D. Wick, Department of Mathematics, Washington University - St. Louis, St. Louis, MO 63130-4899 USA }
\email{wick@math.wustl.edu}

\begin{abstract}
Let $\Delta$ be the Laplace--Beltrami operator acting on a non-doubling manifold with two ends $\mathbb R^m \sharp \mathcal R^n$ with $m > n \ge 3$. 
Let $\frak{h}_t(x,y)$ be the kernels of the semigroup $e^{-t\Delta}$ generated by $\Delta$. We say that a non-negative self-adjoint operator $L$ on $L^2(\mathbb R^m \sharp \mathcal R^n)$ has a heat kernel with upper bound of Gaussian type if the kernel $h_t(x,y)$ of the semigroup $e^{-tL}$ satisfies $ h_t(x,y)  \le C \frak{h}_{\alpha t}(x,y)$ for some constants $C$ and $\alpha$. This class of  operators includes the Schr\"odinger operator $L = \Delta + V$ where $V$ is an arbitrary non-negative potential. We then obtain upper bounds 
of the Poisson semigroup kernel of $L$ together with its time derivatives and use them to show 
the weak type $(1,1)$ estimate for the holomorphic functional calculus $\frak{M}(\sqrt{L})$ where $\frak{M}(z)$ is a function of Laplace transform type.
Our result covers the  purely imaginary powers  $L^{is}, s \in \mathbb R$, as a special case and serves as a model case for weak type $(1,1)$ estimates of singular integrals with non-smooth kernels on non-doubling  spaces.
\end{abstract}




\maketitle{}

\bigskip

\section{Introduction} 

In the last fifty years, the theory of Calder\'on-Zygmund singular integrals has been a central part and  success story 
of modern harmonic analysis. This theory has had extensive influence on other fields of mathematics such as complex analysis 
and partial differential equations. 

Assume that $T$ is a bounded operator on the space $L^2(X)$ 
where $X$ is a metric space with a distance $d$  and a measure $\mu$. Also assume that  $T$ has an 
associated kernel $k(x,y)$ in the sense 
\begin{eqnarray}\label{assockernel} T f(x) = \int _X k(x,y) f(y) d\mu (y) 
\end{eqnarray}
for any continuous function $f$ with compact support and for $x$ not in the support of $f$.


The theory of Calder\'on-Zygmund singular integrals established sufficient conditions on the space $X$ and the associated 
kernel $k(x,y)$ for such an operator $T$ to be bounded on  $L^p(X)$ for $p \ne 2$. There are 2 key conditions:

$\bullet$ Doubling condition: a measure $\mu$ on the metric (or quasi-metric) space $X$
 is said to be doubling if there exists some positive constant $C$ such that
\begin{eqnarray}\label{doubling property}
0<\mu(B(x,2r))\leq C\mu(B(x,r))<+\infty
\end{eqnarray}
for all $x\in X$ and $r>0$, where $B(x,r)$ denotes the ball centered at $x$ and with radius $r>0$.

$\bullet$ H\"ormander condition: the associated kernel $k(x,y)$ is said to satisfy the (almost $L^1$) H\"ormander condition if
there exist positive constants $c$ and $C$ such that
\begin{eqnarray}\label{Hormander condition}
\int_{d(x, y_{1}) \ge c d(y_{1}, y_{2})} | k(x, y_1) - k(x, y_2) | d\mu (x)  \le C
\end{eqnarray}
uniformly of $y_1, y_2$.

Under the doubling condition (\ref {doubling property}) and the H\"ormander condition (\ref{Hormander condition}), 
it is well known that $T$ is of weak type $(1,1)$. By 
Marcinkiewicz interpolation, $T$ is bounded on $L^p(X)$ for $1 < p \le 2$. If the H\"ormander condition \eqref{Hormander condition} is satisfied with $x$ and $y$ swapped, then $T$ is bounded on $L^p(X)$ for $2 \le p <  \infty$.

While the theory of Calder\'on-Zygmund singular integrals has been a great success, there are still many important 
singular integral operators which do not belong to this class. Within  the last twenty years, there were two main directions
of development which study operators beyond this Calder\'on-Zygmund class.

$\bullet$ Singular integrals on non-homogeneous spaces: substantial progress has been made by 
F. Nazarov, S. Treil, A. Volberg, X. Tolsa, T. Hyt\"onen and others in showing that
many features of the classical Calder\'on-Zygmund theory still hold without assuming
the doubling property. More specifically, the doubling condition on $\mathbb{R}^d$ can be
replaced by the polynomial growth condition: for some fixed positive
constants $C$ and $n\in (0,d]$, one has
\begin{eqnarray}\label{growth condition}
\mu(B(x,r))\leq Cr^n \hskip.5cm \textup{for all } x\in \mathbb{R}^d, r>0.
\end{eqnarray}

If the measure $\mu$ satisfies the condition (\ref{growth
condition}), then the space $(\mathbb R^d,\mu)$ is called a
non-homogeneous space. Calder\'on--Zygmund theory has been
developed on such non-homogeneous spaces; see for example \cite{NTV1, NTV2, NTV3, To2}. For the BMO and $H^1$ function space,  the Littlewood--Paley theory,
 and weighted norm inequalities on such non-homogeneous spaces, see 
\cite{BD1,OP,To1,To3};  for Morrey
spaces, Besov spaces and Triebel-Lizorkin spaces in this setting, see \cite{DHY, HY, Sa}. See also \cite{BD1, Hy, HM, HYY} for recent work in this direction which studies a more general setting for non-homogeneous analysis on metric spaces $(X,d,\mu),$ where $(X,d)$ is said to be  geometrically doubling.

However, to obtain boundedness of singular integrals in this setting, one needs certain strong regularity
on the associated kernels in terms of the upper doubling measure, i.e., $r^n$ as in \eqref{growth condition} rather than $\mu(B(x,r))$. For example H\"older continuity on the space variables of the kernels
is needed for weak type $(1,1)$ estimate.

$\bullet$ Singular integrals with non-smooth kernels: A lot of work has been carried out to study singular integrals
whose associated kernels are not smooth enough to satisfy the H\"ormander condition. Substantial progress
has been made by X. Duong, A. McIntosh, S. Hofmann, L. Yan, J. Martell, P. Auscher, T. Coulhon  and others. 
 The H\"ormander condition was replaced by a weaker one to obtain the 
weak type $(1,1)$ estimates and to study function spaces associated with operators. 
See for example \cite{AM, CD1, DMc, DY1, DY2, HLMMY}. 
The achievements in this direction are mostly obtained for operators acting on doubling spaces. 

A natural question arises: How about singular integrals with non-smooth kernels on non-doubling spaces?
This is a difficult  and interesting problem when both the key conditions of Calder\'on-Zygmund theory are missing.
In this paper we study certain singular integrals with non-smooth kernels
acting on non-doubling spaces. Our model here is the holomorphic functional calculus of Laplace transform type
for operators with suitable heat kernel upper bounds such as the Schr\"odinger operator on a non-doubling manifold with two ends.

Let us recall manifolds with ends as in \cite{GS}. Let $M$ be a complete non-compact Riemannian manifold
and $K \subset M$ be a compact set with non-empty interior and smooth boundary such that $M\backslash K$
has $k$ connected components $E_1, \ldots, E_k$. We call $K$ the central part and for simplicity consider the case of two ends, i.e. $k=2$
so that $E_1$ is isometric to $\mathbb R^m$ and $E_2$ is isometric to $\mathcal R^n:=\mathbb R^n \times \mathbb S^{m-n}$
where $m > n \ge 3$ and $ \mathbb S^{m-n}$ is the unit sphere in $\mathbb R^{m-n}$. We denote 
the non-doubling manifold with two ends as $M = \mathbb R^m \sharp \mathcal R^n$ with distance $|x| = \sup_{z\in K} d(x,z)$
where $d(x,z)$ is the geodesic distance in $M$. One can see that $|x|$ is separated from zero in $M$ and 
$|x| \approx 1+ d(x, K)$ where $d(x, K) = \inf \{d(x,y: y\in K \}$.
It is easy to check that  $\mathbb R^m \sharp \mathcal R^n$ is non-doubling since
\begin{equation}\label{eq2s-functiona}
V(x,r) \approx \begin{cases}
r^m \ {\rm for \ all }\  x\in M, {\rm when }\  r \le 1 \\
r^n  \ {\rm for}\  B(x,r) \subset \mathcal R^n, {\rm when }\  r > 1 ; \ {\rm and}\\
r^m \ {\rm for} \ x\in \mathcal R^n \backslash K, \ r >  2|x| , {\rm or } \ x\in \mathbb R^m, r > 1
\end{cases}
\end{equation}
 where $V(x,r)$ is the measure of the ball $B(x,r)$.

In \cite{GS}, Grigor'yan and L. Saloff-Coste studied the kernels of the semigroup $e^{-t\Delta}$  generated by the Laplace-Beltrami
operator $\Delta$ on $\mathbb R^m \sharp \mathcal R^n$  and obtained by probabilistic methods the upper and lower bounds for the  kernels 
$\frak{h}_t(x,y)$ of $e^{-t\Delta}$ . 
However, no further information on $ \frak{h}_t(x,y)$ are known, 
for example we do not know if some (good) pointwise estimates on the time derivatives and space derivatives of $\frak{h}_t(x,y)$ exist. Indeed, the standard method of extending the Gaussian upper bound on the heat kernel with $t > 0$ to complex $z = t+ is$ in the case of doubling space like $\mathbb R^n$ or spaces of homogeneous type does not give a sharp upper bound for the complex heat kernel of the Laplace-Beltrami operator due to the missing of certain (sharp) global estimate such as the $L^1 - L^{\infty}$ estimate of heat semigroup in the setting of  $\mathbb R^m \sharp \mathcal R^n$  .

We view the heat kernel of the Laplace-Beltrami operator on 
$\mathbb R^m \sharp \mathcal R^n$ as the standard behaviour of heat diffusion which plays the important role of the Gaussian kernels on doubling spaces. We introduce the concept of 
heat kernels with upper bounds of Gaussian type as follows.

\begin{definition} Let $\Delta$ be the Laplace-Beltrami operator and $L$ be a non-negative self-adjoint operator on $L^2(\mathbb R^m \sharp \mathcal R^n)$. We say that the heat kernel of $L$ has an upper bound of Gaussian type  if the kernel $h_t(x,y)$ of $e^{-tL}$
	satisfies   $ h_t(x,y)  \le C \frak{h}_{\alpha t}(x,y)$ for some constants $C$ and $\alpha$ where $\frak{h}_t(x,y)$ is   the kernel of $e^{-t\Delta}$.
\end{definition}		
\begin{remark}	(a) See Theorem A, Section 2 for the (sharp) upper bound and lower bound for the heat kernels $\frak{h}_t(x,y)$.

	
	(b) The operators which have heat kernels with upper bounds of Gaussian type include
the Schr\"odinger operator $L = \Delta + V$ where $V$ is  a non-negative potential.
Indeed it follows from the Trotter formula that  
$$e^{-t(\Delta + V)} |f(x)| \le e^{-t\Delta} |f(x)| $$
for $f\in L^2(M)$.  Hence an upper bound for the kernel $\frak{h}_t(x,y)$ of $e^{-t\Delta}$ is also an upper bound for the kernel $h_t(x,y)$ of $e^{-t(\Delta + V)}$.
However, a lower bound for $\frak{h}_t(x,y)$ might not be  a lower bound for  $h_t(x,y)$.
       
(c) There is no assumption on the smoothness of the heat kernel in the definition of upper bound of Gaussian type. In the specific case of the Schr\"odinger operator, due to the effect of the non-negative potential $V$, it is possible that the kernel $h_t(x,y)$ 
of $e^{-t(\Delta + V)}$ is discontinuous
hence regularity estimates such as H\"older continuity are false for $h_t(x,y)$ in general.

\end{remark}

We note that in \cite{DLS}, the authors obtained the weak type $(1,1)$ estimates for the maximal operator $T(f)=\sup_{t > 0} \left\vert e^{-t\Delta}f\right\vert$
by using the upper bounds of the heat kernels $\frak{h}_t(x,y)$. The proof was a direct consequence of the sharp upper bounds on heat kernels 
in \cite{GS}. In \cite{CSY}, the authors obtained some estimates which showed that spectral multipliers for a function $\theta (z)$ with compact support
are bounded on $L^p $ spaces  on a space $X$ which includes the case of non-doubling manifolds with ends. While the result in \cite {CSY} is
applicable to  large class of underlying spaces $X$, the condition that the function $\theta (z)$ having compact support is quite restrictive. Indeed, 
the model case of the function $\theta(z) = z^{is}$, $s$ real, which gives rise to the purely imaginary power $\Delta ^{is}$ 
is not covered by the result  of \cite{CSY}. 
It came to our attention recently that the Riesz transform $\nabla \Delta^{-1/2}$ of the Laplace Beltrami operator $\Delta$ on $\mathbb R^m \sharp \mathcal R^n$ was proved by Carron \cite{Car} to be bounded on $L^p(\mathbb R^m \sharp \mathcal R^n)$ for $p_0 < p < n$ for some $p_0>1$, and after the first version of this article was completed, the $L^p(\mathbb R^m \sharp \mathcal R^n)$ boundedness of the Riesz transform $\nabla \Delta^{-1/2}$ for $1<p<n$, together with the weak type $(1,1)$ estimate, were obtained by Hassell and Sikora \cite{HS}.


The following theorem is our main result.

\begin{theorem}\label{main theorem}
Let  $L $ be an operator which has heat kernel with upper bounds of Gaussian type.
Let $\frak{M}(\sqrt L)$ be the holomorphic functional calculus of Laplace transform type of $\sqrt{L}$ defined by 
$$\frak{M}(\sqrt L)f  = \int_0^\infty [\sqrt L \exp(-t\sqrt L) f ] \tilde{m}(t) dt\, $$ 
in which $\tilde{m}(t)$ is a bounded function on $[0, \infty)$, i.e. 
$ |\tilde{m}(t)|\leq C_0, $ where $C_0$ is a constant.  Then $\frak{M}(\sqrt{L})$ is of weak type $(1,1)$. Hence by interpolation and duality, the operator $\frak{M}(\sqrt{L})$ is bounded on $L^p(\mathbb R^m \sharp \mathcal R^n)$ for $1 < p < \infty$.
\end{theorem}

\begin{remark} (a) In Theorem \ref{main theorem} we prove the weak type $(1,1)$ estimate for $\frak{M}(\sqrt{L})$
for a function $\frak{M}$ of Laplace transform type.  While the 
$L^p$ boundedness of $\frak{M}(\sqrt{\Delta})$ for $1 < p < \infty$ can be obtained by the Littlewood--Paley theory \cite{Stein} or transference method \cite{Cowling}, the end-point weak $(1,1)$ estimate of $\frak{M}(L)$ is new even for the case when 
$L = \Delta$.  Our main result includes the operators $L^{is}$, $s$ real, as a special case
and it is a good example for  singular integrals acting on 
non-doubling spaces whose kernels do not satisfy the H\"ormander condition (\ref{Hormander condition}).

(b) By using the same approach and similar techniques in the proof of our main result, Theorem
\ref{main theorem}, we can also obtain the weak type $(1,1)$ estimate for the Littlewood--Paley square function defined via the Poisson semigroup generated by $L$ as follows:
$$ g(f)(x) = \bigg(\int_0^\infty \big| (t\sqrt{L})^\kappa e^{-t\sqrt{L}}(f)(x) \big|^2 {dt\over t}\bigg)^{1\over 2}, \quad\kappa\in\mathbb N, \kappa\geq 1.$$



 (c) In addition to standard techniques of harmonic analysis of real variables, there are two key elements in our method of proofs in this paper.

(i) Since the pointwise estimates on space and time derivatives of the semigroup  $e^{-tL}$ are not known, we overcome this problem
by using the subordination formula to obtain  upper bounds on the time derivatives of the kernel of Poisson semigroup  $ e^{-t\sqrt{L}}$
via the known upper bound for the kernel of the heat semigroup  $e^{-tL}$. Then we approach the holomorphic functional 
calculus of Laplace transform type $\frak{M}(\sqrt{L})$ through the Poisson semigroup $ e^{-t\sqrt{L}}$.

(ii) The standard Calder\'on--Zygmund decomposition on non-homogeneous spaces (such as \cite{NTV2, To2}) are not applicable to the proof for the weak type $(1,1)$  estimate of our singular integral $\frak{M}(\sqrt{L})$ because of lack of smoothness of its kernel. To overcome this problem, we 
use the technique of generalised approximation to the identity in \cite{DMc} to handle the local doubling part, then carry out a number of subtle decomposition and meticulous estimates to handle the case of non-doubling balls. It  turns out that the sharp upper bounds of the Poisson semigroup are sufficient for us to handle 
the blowing up of non-doubling volumes of balls and obtain the desired weak type $(1,1)$ estimate.

 The method in this paper relies only on good upper bounds on the Poisson semigroup kernel and its time derivatives which
can be derived from the heat semigroup kernels. It does not require, for example the contraction property of the semigroup 
on $L^p$ spaces, hence can be applied to other differential operators. We believe that our method can be developed further to study boundedness of singular integrals with non-smooth kernels acting on non-doubling spaces in other settings.
 
\end{remark}

\section{Poisson semigroup and its time-derivatives}
\setcounter{equation}{0}

Let $\Delta$ be the Laplace Beltrami operator acting on the manifold $\mathbb R^m \sharp \mathcal R^n$ and $\exp(-t\Delta)$ the heat propagator corresponding to $\Delta$. Here and throughout the whole paper, we use $ \mathbb{R}^m\backslash K $ to denote the large end of the manifold $\mathbb R^m \sharp \mathcal R^n$, $ \mathcal{R}^n\backslash K $ to denote the small end, and $K$ to denote the centre part of the manifold.

\begin{thmA}[\cite{GS}]\label{th GS}
The  kernel $\frak{h}_t(x,y)$ of $\exp(-t\Delta)$ satisfies the following estimates:

1. For $t\leq 1$ and all $x,y\in \mathbb R^m \sharp \mathcal R^n$,
$$ \frak{h}_t(x,y)\approx {1\over V(x,\sqrt{t})} \exp\Big(-c_0 {d(x,y)^2\over t} \Big); $$

2. For $t>1$ and all $x,y\in K $,
$$ \frak{h}_t(x,y)\approx {1\over t^{n/2}} \exp\Big(-c_0 {d(x,y)^2\over t} \Big); $$

3. For $t>1$ and  $x\in \mathbb{R}^m\backslash K $, $y\in K$,
$$ \frak{h}_t(x,y)\approx \left({1 \over t^{n/2} |x|^{m-2}}+ {1\over t^{m/2}}\right) \exp\Big(-c_0 {d(x,y)^2\over t} \Big); $$

4. For $t>1$ and  $x\in \mathcal{R}^n\backslash K $, $y\in K$,
$$ \frak{h}_t(x,y)\approx \left({1 \over t^{n/2} |x|^{n-2}}+ {1\over t^{n/2}}\right) \exp\Big(-c_0 {d(x,y)^2\over t} \Big); $$

5. For $t>1$ and  $x\in \mathbb{R}^m\backslash K $, $y\in \mathcal{R}^n\backslash K $,
$$ \frak{h}_t(x,y)\approx \left({1 \over t^{n/2} |x|^{m-2}}+ {1\over t^{m/2}|y|^{n-2}}\right) \exp\Big(-c_0 {d(x,y)^2\over t} \Big); $$

6. For $t>1$ and  $x,y\in \mathbb{R}^m\backslash K $,
$$ \frak{h}_t(x,y)\approx {1 \over t^{n/2} |x|^{m-2}|y|^{m-2}} \exp\Big(-c_0 {|x|^2+|y|^2\over t} \Big) + {1\over t^{m/2}} \exp\Big(-c_0 {d(x,y)^2\over t} \Big); $$

7. For $t>1$ and  $x,y\in \mathcal{R}^n\backslash K $,
$$ \frak{h}_t(x,y)\approx {1 \over t^{n/2} |x|^{n-2}|y|^{n-2}} \exp\Big(-c_0 {|x|^2+|y|^2\over t} \Big) + {1\over t^{n/2}} \exp\Big(-c_0 {d(x,y)^2\over t} \Big). $$

\end{thmA}

We now recall the following result.
\begin{thmB}[\cite{DLS}]
Let $T$ be the maximal operator defined by $T(f)(x):=\sup\limits_{t>0}
|\exp(-t\Delta)f(x)|$. Then $T$ is weak type $(1,1)$ and for any
function $f\in L^p(\mathbb R^m \sharp \mathcal R^n)$, $1<p\leq\infty$, the following estimates hold
$$
\|Tf\|_{L^p(\mathbb R^m \sharp \mathcal R^n)} \le C  \|f\|_{L^p(\mathbb R^m \sharp \mathcal R^n)}.
$$
\end{thmB}

For the rest of the article, let $L$ be a non-negative self-adjoint operator whose heat kernels satisfy upper bounds of Gaussian type.
We first have the following corollary.

\begin{corollary}  Theorem B  holds for the maximal operator via the heat semigroup generated by $L$, i.e.,
$T_L(f)(x):=\sup\limits_{t>0}
|\exp(-tL)f(x)|$ is of weak type $(1,1)$ and bounded on $L^\infty(\mathbb R^m \sharp \mathcal R^n)$, and hence it is bounded on $L^p(\mathbb R^m \sharp \mathcal R^n)$ for all $1<p\leq\infty$.  
\end{corollary}
Proof: This follows directly from the inequality
$$\exp (-tL) | f |(x) \le C \exp (-\alpha t \Delta) | f | (x) $$
and Theorem B.

Next, we study the properties of the Poisson semigroup generated by $L$.  Let $k\in \mathbb{N}$, we denote by $P_{t,k}(x,y)$ the kernel of $ (t\sqrt{L})^k e^{-t\sqrt{L}}$.  For $k=0$, we write $P_t(x,y)$ instead of $P_{t,0}(x,y)$.

\begin{theorem}\label{theorem poisson}
For $k\in \mathbb{N}$, set $k\vee 1=\max\{k,1\}$. Then the exists a constant $C$ (which depends on $k$) such that the kernel  $P_{t,k}(x,y)$ satisfies the following estimates:


1. For $x,y\in K $,
$$ |P_{t,k}(x,y)|\leq \frac{C}{t^m}\Big(\frac{t}{t+d(x,y)}\Big)^{m+k\vee 1}+ \frac{C}{t^n}\Big(\frac{t}{t+d(x,y)}\Big)^{n+k\vee 1}; $$

2. For $x\in \mathbb{R}^m\backslash K $, $y\in K$,
$$ |P_{t,k}(x,y)|\leq \frac{C}{t^m}\Big(\frac{t}{t+d(x,y)}\Big)^{m+k\vee 1} +\frac{C}{t^n|x|^{m-2}}\Big(\frac{t}{t+d(x,y)}\Big)^{n+k\vee 1}; $$

3. For  $x\in \mathcal{R}^n\backslash K $, $y\in K$,
$$ |P_{t,k}(x,y)|\leq \frac{C}{t^m}\Big(\frac{t}{t+d(x,y)}\Big)^{m+k\vee 1}+ \frac{C}{t^n}\Big(\frac{t}{t+d(x,y)}\Big)^{n+k\vee 1}; $$

4. For  $x\in \mathbb{R}^m\backslash K $, $y\in \mathcal{R}^n\backslash K $,
\begin{align*} 
|P_{t,k}(x,y)|&\leq \frac{C}{t^m}\Big(\frac{t}{t+d(x,y)}\Big)^{m+k\vee 1}+ \frac{C}{t^n|x|^{m-2}}\Big(\frac{t}{t+d(x,y)}\Big)^{n+k\vee1}\\
&\quad+\frac{C}{t^m|y|^{n-2}}\Big(\frac{t}{t+d(x,y)}\Big)^{m+k\vee 1}; 
\end{align*}

5. For  $x,y\in \mathbb{R}^m\backslash K $,
$$ |P_{t,k}(x,y)|\leq \frac{C}{t^m}\Big(\frac{t}{t+d(x,y)}\Big)^{m+k\vee 1} +\frac{C}{t^n|x|^{m-2}|y|^{m-2}}\Big(\frac{t}{t+|x|+|y|}\Big)^{n+k\vee1};$$ 

6. For $x,y\in \mathcal{R}^n\backslash K $,
$$ |P_{t,k}(x,y)|\leq \frac{C}{t^m}\Big(\frac{t}{t+d(x,y)}\Big)^{m+k\vee 1}+ \frac{C}{t^n}\Big(\frac{t}{t+d(x,y)}\Big)^{n+k\vee 1}. $$
\end{theorem}

\begin{proof}
By the subordination formula we have
\begin{equation}\label{subor-formula}
e^{-t\sqrt{L}}=\frac{1}{2\sqrt{\pi}}\int_0^\infty\frac{te^{-\frac{t^2}{4v}}}{\sqrt{v}}e^{-vL}\frac{dv}{v},
\end{equation}
and hence
\begin{equation*}
\begin{aligned}
(t\sqrt{L})^{k}e^{-t\sqrt{L}}&=(-1)^k\frac{t^k}{2\sqrt{\pi}}\int_0^\infty\partial^k_t(te^{-\frac{t^2}{4v}})e^{-vL}\frac{du}{v^{3/2}}\\
&=(-1)^k\frac{t^k}{\sqrt{\pi}}\int_0^\infty\partial^{k+1}_t(e^{-\frac{t^2}{4v}})e^{-vL}\frac{dv}{\sqrt{v}}.
\end{aligned}
\end{equation*}
This yields that
\begin{equation}\label{eq1-thm2.1}
\begin{aligned}
P_{t,k}(x,y)&=(-1)^k\frac{t^k}{\sqrt{\pi}}\int_0^\infty\partial^{k+1}_t(e^{-\frac{t^2}{4v}})H_v(x,y)\frac{dv}{\sqrt{v}}
\end{aligned}
\end{equation}
where $H_{v}(x,y)$ is the kernel of $e^{-vL}$.

Let $s>0$ and $k\in \mathbb{N}$. By Fa\`a di Bruno's formula, we can write
$$
\partial^{k+1}_te^{-\frac{t^2}{s}}=\sum\frac{(-1)^{m_1+m_2}}{2 \times m_1!m_2!}e^{-t^2/s}\Big(\frac{t}{s}\Big)^{m_1}\Big(\frac{1}{s}\Big)^{m_2},
$$
where the sum is taken over all pairs $(m_1,m_2)$ of nonnegative integers satisfying $m_1+2m_2=k+1$. For such a pair $(m_1,m_2)$, there exists $C>0$ so that
\begin{align}
e^{-\frac{t^2}{s}}\Big(\frac{t}{s}\Big)^{m_1}\Big(\frac{1}{s}\Big)^{m_2}&=e^{-\frac{t^2}{s}}\Big(\frac{t}{\sqrt{s}}\Big)^{m_1}\Big(\frac{1}{s}\Big)^{m_1/2+m_2}\\
&\leq C e^{-\frac{2t^2}{ 3s}}s^{-(k+1)/2}\max\Big\{1, \Big(\frac{t}{\sqrt{s}}\Big)^{k+1}\Big\}.\nonumber
\end{align}
This implies that
\begin{equation}\label{eq2-thm2.1}
|\partial^{k+1}_te^{-\frac{t^2}{s}}|\leq C e^{-\frac{t^2}{2s}}s^{-(k+1)/2}.
\end{equation}
From \eqref{eq1-thm2.1} and \eqref{eq2-thm2.1} we deduce that
\begin{equation}\label{eq3-thm2.1}
\begin{aligned}
|P_{t,k}(x,y)|&\le  \frac{C}{4\sqrt{\pi}}\int_0^\infty e^{-\frac{t^2}{8v}}\Big(\frac{t}{\sqrt{v}}\Big)^{k} |H_v(x,y)| \frac{dv}{v}.
\end{aligned}
\end{equation}
We now give the estimates for $P_{t,k}(x,y)$ with $k\geq 1$ only, since the remaining case $k=0$ can be done similarly.

We have
\begin{align*}
|P_{t,k}(x,y)|&\le  \frac{C}{4\sqrt{\pi}}\int_0^1 e^{-\frac{t^2}{8v}}\Big(\frac{t}{\sqrt{v}}\Big)^{k} |H_v(x,y)| \frac{dv}{v}+\frac{C}{4\sqrt{\pi}}\int_1^\infty e^{-\frac{t^2}{8v}}\Big(\frac{t}{\sqrt{v}}\Big)^{k} |H_v(x,y)| \frac{dv}{v}\\
&=:\mathbb J_1(x,y)+\mathbb J_2(x,y).
\end{align*}
Applying the upper bound in point 1 in Theorem A for $ |H_v(x,y)| $, we have
\begin{align*} 
\mathbb J_1(x,y)&\leq C\int_0^1   e^{-\frac{t^2}{8v}}\Big(\frac{t}{\sqrt{v}}\Big)^{k}   {1\over v^{m/2}} \exp\Big(- {d(x,y)^2\over cv} \Big) {dv\over v}\\
&\le C\int_0^\infty   \Big(\frac{t}{\sqrt{v}}\Big)^{k}   {1\over v^{m/2}} \exp\Big(-{d(x,y)^2 +t^2\over cv} \Big) {dv\over v}\\
&\le C\left(\int_0^{d(x,y)^2 +t^2} +\int^\infty_{d(x,y)^2 +t^2} \right)  \Big(\frac{t}{\sqrt{v}}\Big)^{k}   {1\over v^{m/2}} \exp\Big(-{d(x,y)^2 +t^2\over cv} \Big) {dv\over v}\\
&\leq \frac{C}{t^m}\Big(\frac{t}{t+d(x,y)}\Big)^{m+k}.
\end{align*}

For the term $\mathbb J_2(x,y)$, we consider the following 6 cases:

Case 1: $x,y\in K $.

Applying the upper bound in point 2 in Theorem A for $ |H_v(x,y)| $,  we have
$$ |p_v(x,y)|\leq {C\over v^{n/2}} \exp\Big(-c_0 {d(x,y)^2\over v} \Big). $$

Arguing similarly to the estimate of $\mathbb J_1(x,y)$ we obtain
\begin{align*} 
\mathbb J_2(x,y)&\leq \frac{C}{t^n}\Big(\frac{t}{t+d(x,y)}\Big)^{n+k}.
\end{align*}

Case 2:  $x\in \mathbb{R}^m\backslash K $, $y\in K$.

Applying the upper bound in point 3 in Theorem A for $ |H_v(x,y)| $,  we get that
$$ |H_v(x,y)|\leq C\left({1 \over v^{n/2} |x|^{m-2}}+ {1\over v^{m/2}}\right) \exp\Big(-c_0 {d(x,y)^2\over v} \Big).$$
Hence, we get
\begin{align*} 
\mathbb J_2(x,y)&\leq C\int_1^\infty   e^{-\frac{t^2}{8v}}\Big(\frac{t}{\sqrt{v}}\Big)^{k}   \left({1 \over v^{n/2} |x|^{m-2}}+ {1\over v^{m/2}}\right) \exp\Big(-c_0 {d(x,y)^2\over v} \Big) {dv\over v}\\
&\leq \frac{C}{t^n|x|^{m-2}}\Big(\frac{t}{t+d(x,y)}\Big)^{n+k}+\frac{C}{t^m}\Big(\frac{t}{t+d(x,y)}\Big)^{m+k}.
\end{align*}

Case 3:   $x\in \mathcal{R}^n\backslash K $, $y\in K$.

Applying the upper bound in point 4 in Theorem A for $ |H_v(x,y)| $,  we have
\begin{align*} 
\mathbb J_2(x,y)&\leq C\int_1^\infty   e^{-\frac{t^2}{8v}}\Big(\frac{t}{\sqrt{v}}\Big)^{k}  \left({1 \over v^{n/2} |x|^{n-2}}+ {1\over v^{n/2}}\right) \exp\Big(-c_0 {d(x,y)^2\over v} \Big){dv\over v}\\
&\leq \frac{C}{t^n|x|^{n-2}}\Big(\frac{t}{t+d(x,y)}\Big)^{n+k}+\frac{C}{t^n}\Big(\frac{t}{t+d(x,y)}\Big)^{n+k}\\
&\leq \frac{C}{t^n}\Big(\frac{t}{t+d(x,y)}\Big)^{n+k}.
\end{align*}

Case 4: $x\in \mathbb{R}^m\backslash K $, $y\in \mathcal{R}^n\backslash K $.

Applying the upper bound in point 5 in Theorem A for $ |H_v(x,y)| $,  we get that
\begin{align*} 
\mathbb J_2(x,y)&\leq C\int_1^\infty   e^{-\frac{t^2}{8v}}\Big(\frac{t}{\sqrt{v}}\Big)^{k}  \left({1 \over v^{n/2} |x|^{m-2}}+ {1\over v^{m/2}|y|^{n-2}}\right) \exp\Big(-c_0 {d(x,y)^2\over v} \Big)  {dv\over v}\\
&\leq \frac{C}{t^n|x|^{m-2}}\Big(\frac{t}{t+d(x,y)}\Big)^{n+k}+\frac{C}{t^m|y|^{n-2}}\Big(\frac{t}{t+d(x,y)}\Big)^{m+k}.
\end{align*}

Case 5: $x,y\in \mathbb{R}^m\backslash K $.

Applying the upper bound in point 6 in Theorem A for $ |H_v(x,y)| $,  we find that
\begin{align*} 
\mathbb J_2(x,y)&\leq C\int_1^\infty   e^{-\frac{t^2}{8v}}\Big(\frac{t}{\sqrt{v}}\Big)^{k}  {1 \over v^{n/2} |x|^{m-2}|y|^{m-2}} \exp\Big(-c_0 {|x|^2+|y|^2\over v} \Big){dv\over v}\\
& \ \ \ \ \  + C\int_1^\infty   e^{-\frac{t^2}{8v}}\Big(\frac{t}{\sqrt{v}}\Big)^{k}  {1\over v^{m/2}} \exp\Big(-c_0 {d(x,y)^2\over v} \Big)  {dv\over v}\\
&\leq \frac{C}{t^n|x|^{m-2}|y|^{m-2}}\Big(\frac{t}{t+|x|+|y|}\Big)^{n+k} +\frac{C}{t^m}\Big(\frac{t}{t+d(x,y)}\Big)^{m+k}.
\end{align*}

Case 6:   $x,y\in \mathcal{R}^n\backslash K $.

Applying the upper bound in point 7 in Theorem A for $ |H_v(x,y)| $, we obtain that
\begin{align*} 
\mathbb J_2(x,y)&\leq C\int_1^\infty   e^{-\frac{t^2}{8v}}\Big(\frac{t}{\sqrt{v}}\Big)^{k}  {1 \over v^{n/2} |x|^{n-2}|y|^{n-2}} \exp\Big(-c_0 {|x|^2+|y|^2\over v} \Big){dv\over v}\\
& \ \ \ \ \ \  + \int_1^\infty   e^{-\frac{t^2}{8v}}\Big(\frac{t}{\sqrt{v}}\Big)^{k}{1\over v^{n/2}} \exp\Big(-c_0 {d(x,y)^2\over v} \Big)  {dv\over v}\\
&\leq \frac{C}{t^n|x|^{n-2}|y|^{n-2}}\Big(\frac{t}{t+d(x,y)}\Big)^{n+k}+\frac{C}{t^n}\Big(\frac{t}{t+d(x,y)}\Big)^{n+k}\\
&\leq \frac{C}{t^n}\Big(\frac{t}{t+d(x,y)}\Big)^{n+k}.
\end{align*}

\end{proof}

We observe that the proof of Theorem \ref{theorem poisson}  can be extended to obtain the estimates
for the complex Poisson semigroup and its time derivatives $(z\sqrt{L})^k\exp(-z\sqrt{L})$.
Indeed, we have the following result.

\begin{theorem}\label{complexPoisson}
Fix $0 < \mu < {\pi \over 4}$, let $S_{\mu}^{0} = \{ z \in \mathbb C: |\arg z | < \mu \}$ and choose $z \in S_{\mu}^{0}$. 
The complex Poisson semigroup and its time derivatives $(z\sqrt{L})^k\exp(-z\sqrt{L})$ exist and satisfy the upper bounds
as in Theorem \ref{theorem poisson} with $t$ to be replaced by $|z|$.
\end{theorem}
\begin{proof}
Fix $0 < \mu < {\pi \over 4}$. For
$z \in S_{\mu}^{0}$, define 
\begin{equation}\label{complexextension}
e^{-z\sqrt{L}}=\frac{1}{2\sqrt{\pi}}\int_0^\infty\frac{ze^{-\frac{z^2}{4v}}}{\sqrt{v}}e^{-vL}\frac{dv}{v}.
\end{equation}
For $0 < |\arg z| < \mu < {\pi \over 4}$ we have $\Re z^2 > 0$,  hence the integral in \eqref{complexextension} converges. When $z$ is real, formula \eqref{complexextension} coincides with the Poisson semigroup \eqref{subor-formula}, hence formula \eqref{complexextension} defines the complex Poisson semigroup (which is unique by analyticity).  Hence
\begin{equation*}
\begin{aligned}
(z\sqrt{L})^{k}e^{-z\sqrt{L}}&=(-1)^k\frac{z^k}{2\sqrt{\pi}}\int_0^\infty\partial^k_z(ze^{-\frac{z^2}{4v}})e^{-vL}\frac{du}{v^{3/2}}\\
&=(-1)^k\frac{z^k}{\sqrt{\pi}}\int_0^\infty\partial^{k+1}_z(e^{-\frac{z^2}{4v}})e^{-vL}\frac{dv}{\sqrt{v}}.
\end{aligned}
\end{equation*}
This yields that
\begin{equation}\label{eq1-thm2.3}
\begin{aligned}
P_{z,k}(x,y)&=(-1)^k\frac{z^k}{\sqrt{\pi}}\int_0^\infty\partial^{k+1}_z(e^{-\frac{z^2}{4v}})H_v(x,y)\frac{dv}{\sqrt{v}}
\end{aligned}
\end{equation}
where $H_{v}(x,y)$ is the kernel of $e^{-vL}$. The rest of the proof is similar to Theorem \ref{theorem poisson}.
\end{proof}

We now obtain a weak type $(1,1)$ estimate for a maximal operator.
\begin{prop}\label{theorem2}
Fix $0 < \mu < {\pi \over 4}$. Let $T_k$ be the operator defined by 
$$T_k(f)(x) := \sup\limits_{z\in S_{\mu}^0} |(z\sqrt{L})^k\exp(-z\sqrt{L})f(x)|$$
 for an integer $k\geq0$ and $f \in L^p(\mathbb R^m \sharp \mathcal R^n)$. Then $T_k$ is of weak type $(1,1)$ and bounded on $L^p(\mathbb R^m \sharp \mathcal R^n)$ for  $1<p\leq\infty$. 
 \end{prop}
\begin{proof}
We point out that with the upper bound of the kernel $P_{z,k}(x,y)$ of $(z\sqrt{L})^k\exp(-z\sqrt{L})$, the 
 weak type $(1,1)$ estimate of the maximal operator $T_k$ follows from the same idea and approach in 
 proof of Theorem B. For more details, we refer to \cite{DLS}.
\end{proof}

The concept of approximations to the identity plays an important role in harmonic analysis. For a family of approximations
to the identity $\phi_t \ast f$ in a doubling space like $\mathbb R^n$, the upper bound on $\phi (x)$ can be taken as the Gaussian bound
with exponential decay or Poisson bound with polynomial decay. However, in a non-doubling space like a 
manifold with ends $\mathbb R^m \sharp \mathcal R^n$, it is not obvious which type of bound is deemed natural. Here we 
suggest to use the Poisson kernels in the definition of an approximation to the identity in this setting. We note that 
in the case of $\mathbb R^n$, the term
${\displaystyle {1\over t^n}}$ in the Poisson kernel ${\displaystyle {1\over t^n} \times {c_n \over (1+ |x-y|^2)^{{n+1\over 2}}}}$ is independent of 
$x$ and $y$, whereas the corresponding  term in the case of $\mathbb R^m \sharp \mathcal R^n$ might depend on $x$ and $y$.

\begin{definition} \label{approximation} A family of kernels $\phi_t (x,y)$, $t> 0$, is said to be a generalised approximation to  the identity if $| \phi_t (x,y) |$
has the same upper bound as $C P_{\alpha t,k}(x,y) $ in Theorem \ref{theorem poisson}
for some positive constants $C, k$ and $\alpha$.
\end{definition}

We note that in the proof of our main result, Theorem \ref{main theorem}, we use $e^{-t\sqrt{L}}$ as a generalised approximation to the identity.
While it is true that $e^{-t\sqrt{L}}$ tends to the Identity as $t $ tends to $0$ in $L^2$ sense, we do not need this property in our proof.

The following result is similar to the basic result in $\mathbb R^n$ 
that the operator $\sup_{t >0} \left\vert \phi_t \ast f\right\vert$ is bounded on $L^p$, $1 < p < \infty$ for a suitable  family of kernels $\phi_t$.

\begin{prop}\label{theorem2bis} Assume that $\phi_t (x,y)$, $t> 0$, is  a generalised approximation to the identity on 
$\mathbb R^m \sharp \mathcal R^n$. Define the family of operators $D_t$ by
$$D_t f(x) = \int_{\mathbb R^m \sharp \mathcal R^n} \phi_t (x,y) f(y) \ d\mu (y)$$
for $f \in L^p(\mathbb R^m \sharp \mathcal R^n)$, $1 < p < \infty$.  
Then the operator $T (f)(x) := \sup\limits_{t>0}|D_t f(x)|$ is bounded on $L^p(\mathbb R^m \sharp \mathcal R^n)$, $1 < p < \infty$ and is of weak type $(1,1)$.
\end{prop}
{\it Sketch of proof:}
We note that  by Definition \ref{approximation} , $D_t | f |(x) $ has the same upper bound as $C (\alpha t \sqrt{L})^k e^{-\alpha t \sqrt{L}} | f |(x) $,
hence the proof follows the same line as that of Proposition \ref{theorem2}.

 \begin{remark}
 Theorem \ref{theorem poisson}, Propositions \ref{theorem2} and \ref{theorem2bis} are of independent interest as they are
 useful for the study of harmonic analysis on the setting of non-doubling manifolds with ends.
 \end{remark}

\section{Proof of main result: Theorem \ref{main theorem}}
\setcounter{equation}{0}

To begin with, we first recall the standard definition of the maximal function and its properties. For any  $p \in [1,\infty]$ and any function $f\in L^p$
we set
$$
\M f(x) =\sup \left\{ \frac{1}{|B(y,r)|}\int_{B(y,r)}|f(z)|dz\colon x\in B(y,r) \right\}.
$$
\begin{thmC}[\cite{DLS}]\label{max}
The maximal function operator is of weak type $(1,1)$ and bounded on all $L^p$ spaces for $1<p\leq \infty$.
\end{thmC}

Now to prove Theorem \ref{main theorem},
it suffices to show that there exists a positive constant $C$ such that for $f\in L^1(\mathbb R^m \sharp \mathcal R^n)$ and for every $\lambda>0$,  
\begin{align}\label{weak11}
|\{ x\in \mathbb R^m \sharp \mathcal R^n: |\frak{M}(\sqrt L)f (x) |>\lambda \}|\leq C{\|f\|_{L^1(\mathbb R^m \sharp \mathcal R^n)}\over \lambda}.
\end{align} 
Then, to prove \eqref{weak11}, it suffices to verify the following three inequalities:
\begin{align}\label{e1}
|\{ x\in \mathbb{R}^m\backslash K: |\frak{M}(\sqrt L)f (x) |>\lambda \}|\leq  C{\|f\|_{L^1(\mathbb{R}^m\, \sharp\, \mathbb{R}^n)}\over \lambda},
\end{align} 
\begin{align}\label{e2}
|\{ x\in \mathcal{R}^n\backslash K: |\frak{M}(\sqrt L)f (x) |>\lambda \}|\leq  C{\|f\|_{L^1(\mathbb{R}^m\, \sharp\, \mathbb{R}^n)}\over \lambda},
\end{align} 
and\begin{align}\label{e3}
|\{ x\in K: |\frak{M}(\sqrt L)f (x) |>\lambda \}|\leq  C{\|f\|_{L^1(\mathbb{R}^m\, \sharp\, \mathbb{R}^n)}\over \lambda}.
\end{align}

We now set 
$$ f_1(x):=f(x)\chi_{\mathbb{R}^m\backslash K},\  f_2(x):=f(x)\chi_{\mathcal{R}^n\backslash K},\ {\rm and}\ f_3(x):=f(x)\chi_{K}.   $$
Thus, $f$ can be written as 
$$ f=f_1+f_2+f_3. $$

Since $\frak{M}(\sqrt L)$ is a linear operator, the measure in the left-hand side of \eqref{e1} satisfies
\begin{align*}
|\{ x\in \mathbb{R}^m\backslash K: |\frak{M}(\sqrt L)f (x) |>\lambda \}|
&\leq \left|\left\{ x\in \mathbb{R}^m\backslash K: |\frak{M}(\sqrt L)f_1 (x) |> {\lambda\over3} \right\}\right|\\
&\quad+ \left|\left\{ x\in \mathbb{R}^m\backslash K: |\frak{M}(\sqrt L)f_2 (x) |>{\lambda\over3} \right\}\right|\\
&\quad\quad+\left|\left\{ x\in \mathbb{R}^m\backslash K: |\frak{M}(\sqrt L)f_3 (x) |>{\lambda\over3} \right\}\right|\\
&=: I_1+I_2+I_3.
\end{align*} 
Similarly, we decompose and obtain the measure of the set $\{ x\in \mathcal{R}^n\backslash K: |\frak{M}(\sqrt L)f (x) |>\lambda \}$ is
 bounded by the sum of three parts
$II_1 +  II_2 + II_3$ and  the measure of $\{ x\in  K: |\frak{M}(\sqrt L )f (x) |>\lambda \}$ is bounded by the sum
$III_1+III_2+III_3$ which correspond to the components $f_1, f_2$ and $f_3$ respectively.


It then suffices to prove that each of the terms above has an upper estimate of the form $C\frac{\left\| f\right\|_{L^{1}(\mathbb R^m \sharp \mathcal R^n)}}{\lambda}$.

\subsection{Estimate of $I_1$}


In this case, since $x$ is in $\mathbb{R}^m\backslash K$ and the function $f_1$ is also supported in $\mathbb{R}^m\backslash K$,
we can restrict to the setting $\mathbb{R}^m\backslash K$, where the measure now becomes the standard Lebesgue measure on $\mathbb{R}^m\backslash K$ which is doubling.  However, the non-homogeneous property shows up in the kernel estimate in this case.
 The Poisson kernel here is not bounded by the classical upper bound and the main difficulty comes from the term
$$\frac{C}{t^n|x|^{m-2}|y|^{m-2}}\Big(\frac{t}{t+|x|+|y|}\Big)^{n+k\vee1}$$
in the upper bound of the Poisson kernel, where the power of the time scaling is $n$ which can be much smaller than the space dimension $m$ while we have extra decay from the terms $|x|^{m-2}$ and $|y|^{m-2}$. Hence, the new method here is to have a refined classification 
of the dyadic cubes (see $\mathcal I_1$ and $\mathcal I_2$ below in the proof) such that 
for most of the cubes (see the term $I_{122}$ below), the terms $|x|^{m-2}$ and $|y|^{m-2}$ can provide suitable decay that
makes a compensation of the lack of power of the time scaling $t$ and that for the rest of the cubes (see the term $I_{123}$ below), the kernel of $m(\sqrt L)$ itself has proper decay which enables the weak type estimate holds.

To begin the proof, we now restrict the setting to $\mathbb{R}^m\backslash K$, and $f_1 $ is in $L^1(\mathbb{R}^m\backslash K)$.
Extend $f_1$ to the whole of $\mathbb{R}^m$ by zero extension, i.e., define $f_1(x):=0$ when $x\in K$.


%
%

We now consider the standard Calder\'on--Zygmund decomposition as follows. Recall that the standard dyadic cubes in 
$\mathbb{R}^m$ are of the form
$$[2^{k}a_1, 2^k(a_1+1))\times \cdots\times [2^{k}a_m, 2^k(a_m+1)),$$
where $k,a_1,\ldots,a_m$ are integers.  Decompose $\mathbb{R}^m$ into a mesh of equal size disjoint dyadic cubes so that
$$ |Q|\geq {1\over \lambda}\|f_1\|_{L^1(\mathbb{R}^m)} $$
for every cube in the mesh. Subdivide each cube in the mesh into $2^m$ congruent cubes by bisecting each of its sides. We now have a new mesh of dyadic cubes. Select a cube in the new mesh if
\begin{align}\label{stopping}
 {1\over |Q|}\int_Q |f_1(x)|dx >\lambda.
\end{align}
Let $\mathcal S$ be the set of all these selected cubes. Now subdividing each non-selected cube into $2^m$ congruent subcubes by bisecting each side as before. Then select one of these new cubes if \eqref{stopping} holds. Put all these selected cubes of this generation into the set $\mathcal S$. Repeat this procedure indefinitely.

Then we have $\mathcal S=\cup_j Q_j$, where all these $Q_j's$ are disjoint, and we further have
$$ |S|=\sum_j |Q_j| \leq {1\over \lambda} \sum_j \int_{Q_j} |f_1(x)|dx \leq {1\over \lambda} \|f_1\|_{L^1(X)}. $$
Define
$$ b_j(x):=\bigg( f_1(x)- {1\over |Q_j|}\int_{Q_j}f_1(y) dy \bigg)\chi_{Q_j}(x)$$
and
$$ b(x):=\sum_j b_j(x), \quad g(x):=f_1(x)-b(x).$$
For a selected $Q_j$, there exists a unique non-selected dyadic cube $Q'$ with twice its side length that contains $Q_j$. Since
$Q'$ is not selected, we get that
$$ {1\over |Q'|}\int_{Q'} |f_1(y)|dy \leq \lambda, $$
which implies that
$$ {1\over |Q_j|}\int_{Q_j}|f_1(y)| dy\leq {2^m\over |Q'|}\int_{Q'} |f(y)|dy \leq 2^m\lambda. $$
For the good part $g(x)$, since $b=0$ on $F:=\mathbb{R}^m\backslash S$, we have
$$ g(x)=f_1(x)\ {\rm on}\ F,\quad{\rm and}\quad g(x)={1\over |Q_j|}\int_{Q_j} f_1(x)dx \ {\rm on}\ Q_j.$$
Then it is easy to verify that $$\|g\|_{L^1(\mathbb{R}^m)}\leq \|f\|_{L^1(\mathbb{R}^m)}\quad{\rm and}\quad \|g\|_{L^\infty(\mathbb{R}^m)}\leq C\lambda.$$

We now have
\begin{align*}
I_1&\leq \left|\left\{ x\in \mathbb{R}^m\backslash K: |\frak{M}(\sqrt L )g (x) |> {\lambda\over6} \right\}\right|\\
&\quad+  \left|\left\{ x\in (\mathbb{R}^m\backslash K) \backslash \cup_i 8Q_i: \big|\frak{M}(\sqrt L ) \big(\sum_j b_{j}\big)  (x) \big|>{\lambda\over6} \right\}\right|\\
&\quad+ | \cup_i 8Q_i|\\
&=: I_{11}+I_{12}+I_{13}. 
\end{align*} 

As for $I_{11}$, by using the $L^2$ boundedness of $\frak{M}(\sqrt L )$, we obtain that
\begin{align*}
I_{11}&= \left|\left\{ x\in \mathbb{R}^m\backslash K: |\frak{M}(\sqrt L )g (x) |>{\lambda\over6} \right\}\right|\\
&\leq {C\over \lambda^2} \|g\|_{L^2(\mathbb{R}^m\backslash K)}^2\\
&\leq {C\over \lambda} \|f\|_{L^1(\mathbb{R}^m\sharp\mathcal{R}^n)},
\end{align*} 
where we use the fact that $|g(x)|\leq C\lambda$.


As for $I_{13}$, note that we have the doubling condition in this case. So we get that
\begin{align*}
I_{13}\leq  C \sum_i |Q_i|\leq  C{\|f\|_{L^1(\mathbb{R}^m\sharp\mathcal{R}^n)}\over \lambda}.
\end{align*}

As for $I_{12}$, we now split all the $Q_i$'s in the set $\mathcal S$ into two groups: 
$$\mathcal{I}_1:=\{i: \text{none of the corners of  $Q_i$ is the origin}\},$$ 
and 
$$\mathcal{I}_2=\{i: \text{one of the corners of  $Q_i$ is the origin}\}.
$$
Write
$$
 \begin{aligned}
 \frak{M}(\sqrt L )\big(\sum_i b_{i}\big)(x) = \sum_{i\in \mathcal{I}_1} \frak{M}(\sqrt L ) b_{i} (x)+\sum_{i\in \mathcal{I}_2} \frak{M}(\sqrt L ) b_{i} (x).
 \end{aligned} 
 $$
For each $i\in \mathcal{I}_1$, we further decompose
$$ \frak{M}(\sqrt  L )b_{i} (x) = \frak{M}(\sqrt L ) e^{-t_i\sqrt{ L }} b_{i}  (x)  + \frak{M}(\sqrt L )\big(I- e^{-t_i\sqrt{ L }}\big) b_{i}  (x), $$
where $\{e^{-t\sqrt{ L }}\}_{t>0}$ is the Poisson semigroup of $L$ as studied in Section 2, and for each $i$, $t_i$ is the side length of the cube $Q_i$.

 Then we have
\begin{align*}
I_{12}&\leq  \left|\left\{ x\in (\mathbb{R}^m\backslash K) \backslash \cup_i 8Q_i: \big|\frak{M}(\sqrt L ) \big(\sum_{i\in \mathcal{I}_1} e^{-t_i\sqrt{ L }} b_{i}\big)  (x) \big|>{\lambda\over18} \right\}\right|\\
&\quad+ \left|\left\{ x\in (\mathbb{R}^m\backslash K) \backslash \cup_i 8Q_i: \Big|\frak{M}(\sqrt L ) \Big(\sum_{i\in\mathcal{I}_1} \big(I- e^{-t_i\sqrt{ L }}\big) b_{i}\Big)  (x) \Big|>{\lambda\over18} \right\}\right|\\
&\quad+\left|\left\{ x\in (\mathbb{R}^m\backslash K) \backslash \cup_i 8Q_i: \big|\frak{M}(\sqrt L ) \big(\sum_{i\in \mathcal{I}_2} b_{i}\big)  (x) \big|>{\lambda\over18} \right\}\right|\\
&=: I_{121}+I_{122}+I_{123}. 
\end{align*}

We first estimate $ I_{121}$. To see this, we claim that
\begin{align}\label{claim121}
\Big\|\sum_{i\in \mathcal{I}_1} e^{-t_i\sqrt{L}} b_{i}\Big\|_{L^2(\mathbb{R}^m\sharp\mathcal{R}^n)}
\leq  C\lambda^{1\over2}\|f_1\|_{L^1(\mathbb{R}^m\sharp\mathcal{R}^n)}^{1\over2}.
\end{align}
To verify this claim, it suffices to show the following 3 cases:
\begin{align}\label{claim121case1}
\Big\|\sum_{i\in \mathcal{I}_1} e^{-t_i\sqrt{L}} b_{i}\Big\|_{L^2(\mathbb{R}^m\backslash K)}
\leq  C\lambda^{1\over2}\|f_1\|_{L^1(\mathbb{R}^m\sharp\mathcal{R}^n)}^{1\over2},
\end{align}
\begin{align}\label{claim121case2}
\Big\|\sum_{i\in \mathcal{I}_1} e^{-t_i\sqrt{L}} b_{i}\Big\|_{L^2(\mathcal{R}^n\backslash K)}
\leq  C\lambda^{1\over2}\|f_1\|_{L^1(\mathbb{R}^m\sharp\mathcal{R}^n)}^{1\over2},
\end{align}
and
\begin{align}\label{claim121case3}
\Big\|\sum_{i\in \mathcal{I}_1} e^{-t_i\sqrt{L}} b_{i}\Big\|_{L^2(K)}
\leq  C\lambda^{1\over2}\|f_1\|_{L^1(\mathbb{R}^m\sharp\mathcal{R}^n)}^{1\over2}.
\end{align}
Hence, combining the estimates of \eqref{claim121case1}, \eqref{claim121case2} and \eqref{claim121case3}, for $I_{121}$, we get that
\begin{align*}
I_{121}&\leq  \left|\left\{ x\in (\mathbb{R}^m\backslash K) \backslash \cup_i 8Q_i: \big|\frak{M}(\sqrt{L}) \big(\sum_{j\in\mathcal{I}_1} e^{-t_j\sqrt{L}} b_{j}\big)  (x) \big|>{\lambda\over18} \right\}\right|\\
&\leq {C\over \lambda^2} \Big\|\sum_{i\in\mathcal{I}_1} e^{-t_i\sqrt{L}} b_{i}\Big\|_{L^2(\mathbb{R}^m\sharp\mathcal{R}^n)}^2\\
&\leq {C\over \lambda} \|f_1\|_{L^1(\mathbb{R}^m\sharp\mathcal{R}^n)}.
\end{align*}

We first estimate \eqref{claim121case1}.  Consider the function $ e^{-t_i\sqrt{L}} b_{i}  (x)$ for $x\in \mathbb{R}^m\backslash K$.
Since 
$$ e^{-t_i\sqrt{L}} b_{i}  (x) = \int_{\mathbb{R}^m\backslash K} P_{t_i}(x,y) b_{i} (y)dy, $$
applying the upper bound in point 5 in Theorem \ref{theorem poisson} for $ |P_{t_i}(x,y)| $ we obtain that
\begin{align*}
|e^{-t_i\sqrt{ L}} b_{i}  (x)| &\leq \int_{\mathbb{R}^m\backslash K} |P_{t_i}(x,y)|\, |b_{i} (y)|dy\\
&\leq C\int_{\mathbb{R}^m\backslash K} \bigg({t_i\over |x|^{m-2}|y|^{m-2}(t_i+|x|+|y|)^{n+1}}+{t_i\over  (t_i+d(x,y))^{m+1}}\bigg)\, |b_{i}(y)|dy\\
&= C\int_{Q_i} {t_i\over |x|^{m-2}|y|^{m-2}(t_i+|x|+|y|)^{n+1}}\, |b_{i} (y)|dy\\
&\quad+ C\int_{\mathbb{R}^m\sharp\mathcal{R}^n} {t_i\over  (t_i+d(x,y))^{m+1}}\, |b_{i} (y)|dy\\
&=: F_{1,i} + F_{2,i}.
\end{align*}

We turn to estimating the term $F_{2,i}$.  In this case, we have $x\in \mathbb{R}^m\backslash K$ and $Q_i\subset \mathbb{R}^m\backslash K$, dyadic, with none of the corners of  $Q_i$ being the origin.
This implies that
$$\sup_{z\in Q_i} {t_i\over  (t_i+d(x,z))^{m+1}} \le C \inf_{z\in Q_i} {t_i\over  (t_i+d(x,z))^{m+1}} .$$
Therefore
\begin{align*}
F_{2,i}&\leq C\sup_{z\in Q_i} {t_i\over  (t_i+d(x,z))^{m+1}} \int_{\mathbb{R}^m\sharp\mathcal{R}^n} \, |b_{i}(y)|dy\\
&\leq C\inf_{z\in Q_i} {t_i\over  (t_i+d(x,z))^{m+1}}\ \lambda |Q_i|\\
&\leq  C\lambda\int_{\mathbb{R}^m\sharp\mathcal{R}^n} {t_i\over  (t_i+d(x,z))^{m+1}}\chi_{Q_i}(z)dz,
\end{align*}
where $\chi_{Q_i}$ is the characteristic function of $Q_i$.

%
For any $h\in L^2(\mathbb{R}^m\backslash K)$ with $\|h\|_{L^2(\mathbb{R}^m\backslash K)}=1$, we get that
\begin{align*}
\langle F_{2,i}, h\rangle
&= C\lambda \int_{\mathbb{R}^m\sharp\mathcal{R}^n}  \int_{\mathbb{R}^m\backslash K} {t_i\over  (t_i+d(x,z))^{m+1}} h(x)dx\ \chi_{Q_i}(z)dz \\
&\leq C\lambda \langle \mathcal M(h), \chi_{Q_i}\rangle.
\end{align*}
As a consequence we obtain that
$$ \Big\langle \sum_{i\in \mathcal{I}_1}  F_{2,i}, h\Big\rangle \leq C\lambda \Big\langle \mathcal M(h),  \sum_{i\in \mathcal{I}_1}  \chi_{Q_i}\Big\rangle, $$
which yields
\begin{align*}
\Big\|\sum_{i\in \mathcal{I}_1} F_{2,i}\Big\|_{L^2(\mathbb{R}^m\backslash K)} &\leq C\lambda \Big\|\sum_{i\in \mathcal{I}_1} \chi_{Q_i}\Big\|_{L^2(\mathbb{R}^m\backslash K)} \leq C\lambda  \bigg(\sum_{i\in \mathcal{I}_1} |Q_i|\bigg)^{1/2}\\
&\leq  C\lambda {\|f_1\|_{L^1(\mathbb{R}^m\sharp\mathcal{R}^n)}^{1\over2}\over \lambda^{1\over2}}\\
&\leq  C\lambda^{1\over2} \|f_1\|_{L^1(\mathbb{R}^m\sharp\mathcal{R}^n)}^{1\over2}.
\end{align*}

To handle $F_{1,i}$, we note that the distance of $Q_i$ to the center $K$ is comparable to the side length of $Q_i$ for $i\in \mathcal{I}_1$ since none of the corners of  $Q_i$ are the origin.  Hence, we obtain that 
\begin{equation}
\label{supinf Q}
\sup_{z\in Q_i}|z| \approx \inf_{z\in Q_i}|z|.
\end{equation} Thus, 
we further obtain that
\begin{align*}
F_{1,i}&\leq C\sup_{z\in Q_i} {t_i\over |x|^{m-2}|z|^{m-2}(t_i+|x|+|z|)^{n+1}}\ \int_{\mathbb{R}^m\sharp\mathcal{R}^n} \, |b_{i}(y)|dy\\
&\leq C\inf_{z\in Q_i} {t_i\over |x|^{m-2} |z|^{m-2}(t_i+|x|+|z|)^{n+1}} \lambda |Q_i|\\
&\leq  C\lambda\int_{\mathbb{R}^m\sharp\mathcal{R}^n} {t_i\over |x|^{m-2}|z|^{m-2}(t_i+|x|+|z|)^{n+1}}\chi_{Q_i}(z)dz.
\end{align*}
Consequently, for any $h\in L^2(\mathbb{R}^m\backslash K)$ with $\|h\|_{L^2(\mathbb{R}^m\backslash K)}=1$,
\begin{align*}
|\langle F_{1,i}, h\rangle|
&\leq C\lambda \int_{\mathbb{R}^m\backslash K}  \int_{\mathbb{R}^m\backslash K}   {t_i\over |x|^{m-2}|z|^{m-2}(t_i+|x|+|z|)^{n+1}}|h(x)|dx\  \chi_{Q_i}(z)dz\\
&\leq C \lambda \int_{\mathbb{R}^m\backslash K}  \mathcal G(h)(z)  \chi_{Q_i}(z)dz,
\end{align*}
where $\mathcal G$ is an operator defined as 
$$\mathcal  G(h)(z) :=  \int_{\mathbb{R}^m\backslash K}   {1\over |x|^{m-2}|z|^{m-2}(|x|+|z|)^{n}}|h(x)|dx.$$
Next, it is direct to see that $\mathcal G$ is a bounded operator on $L^2(\mathbb{R}^m\backslash K)$:
\begin{align*}
\|\mathcal  G(h)\|_{L^2(\mathbb{R}^m\backslash K)}^2
& \leq \int_{\mathbb{R}^m\backslash K}\int_{\mathbb{R}^m\backslash K}   {1\over |x|^{2m-4}|z|^{2m-4}(|x|+|z|)^{2n}}dx  \|h\|_{L^2(R^m\backslash K)}^2 dz\\
& \leq \|h\|_{L^2(\mathbb{R}^m\backslash K)}^2\int_{\mathbb{R}^m\backslash K}\int_{\mathbb{R}^m\backslash K}   {1\over |x|^{2m-4}|z|^{2m-4}|x|^{n}|z|^{n}}dx  dz\\
& \leq \|h\|_{L^2(\mathbb{R}^m\backslash K)}^2\int_{\mathbb{R}^m\backslash K}   {1\over |x|^{2m-4}|x|^{n}}dx  \, \int_{\mathbb{R}^m\backslash K}   {1\over |z|^{2m-4}|z|^{n}}  dz\\
&\leq C\|h\|_{L^2(\mathbb{R}^m\backslash K)}^2,
\end{align*}
where in the last inequality we use the condition that $m>n>2$.

As a consequence, similar to the estimates for $\sum_{i\in \mathcal{I}_1} F_{2,i}$, we obtain that
\begin{align*}
\Big\|\sum_{i\in \mathcal{I}_1} F_{1,i}\Big\|_{L^2(\mathbb{R}^m\backslash K)} &\leq C\lambda 
\Big\|\sum_{i\in \mathcal{I}_1} \chi_{Q_i}\Big\|_{L^2(\mathbb{R}^m\backslash K)} \\
&\leq C\lambda  \bigg(\sum_i |Q_i|\bigg)^{1/2}\\
&\leq  C\lambda^{1\over2}\|f_1\|_{L^1(\mathbb{R}^n\sharp\mathbb{R}^m)}^{1\over2}.
\end{align*}

Combining the estimates with respect to $F_{1,i}$ and $F_{2,i}$ above, we deduce that
\begin{align*}
\Big\|\sum_{i\in \mathcal{I}_1} e^{-t_i\sqrt{L}} b_{i}\Big\|_{L^2(\mathbb{R}^m\backslash K)}
&\leq \Big\|\sum_{i\in \mathcal{I}_1} F_{1,i}\Big\|_{L^2(\mathbb{R}^m\backslash K)}+\Big\|\sum_{i\in \mathcal{I}_1} F_{2,i}\Big\|_{L^2(\mathbb{R}^m\backslash K)}\\
&\leq  C\lambda^{1\over2}\|f_1\|_{L^1(\mathbb{R}^m\sharp\mathcal{R}^n)}^{1\over2},
\end{align*}
which shows that
the claim \eqref{claim121case1} holds.

\smallskip

We now estimate \eqref{claim121case2}.
Consider the function $ e^{-t_i\sqrt{\Delta}} b_{i}  (x)$ for $x\in \mathcal{R}^n\backslash K$.
Since 
$$ e^{-t_i\sqrt{\Delta}} b_{i}  (x) = \int_{\mathbb{R}^m\backslash K} P_{t_i}(x,y) b_{i} (y)dy, $$
applying the upper bound in point 4 in Theorem \ref{theorem poisson} for $ |P_{t_i}(x,y)| $ we obtain that
 \begin{align*}
&|e^{-t_i\sqrt{ L}} b_{i}  (x)|\\
 &\leq \int_{\mathbb{R}^m\sharp\mathcal{R}^n} |P_{t_i}(x,y)|\, |b_{i} (y)|dy\\
&\leq C\int_{\mathbb{R}^m\sharp\mathcal{R}^n} \bigg(\frac{t_i}{(t_i+d(x,y))^{m+1}}+ \frac{1}{|x|^{n-2}}\frac{t_i}{(t_i+d(x,y))^{m+1}}+\frac{1}{|y|^{m-2}}\frac{t}{(t_i+d(x,y))^{n+1}}
\bigg)\, |b_{i}(y)|dy\\
&\leq C\int_{\mathbb{R}^m\sharp\mathcal{R}^n} \frac{t_i}{(t_i+d(x,y))^{m+1}} |b_{i}(y)|dy +C\int_{\mathbb{R}^m\sharp\mathcal{R}^n} \frac{1}{|y|^{m-2}}\frac{t_i}{(t_i+d(x,y))^{n+1}}
 |b_{i}(y)|dy\\
&=: G_{1,i} + G_{2,i}
\end{align*}
where the third inequality follows from the fact that $|x| \ge 1$, hence the second term in the integrand is dominated by the first term.

By the equivalence in \eqref{supinf Q}, we have
\begin{align*}
G_{2,i}&\leq C\int_{\mathbb{R}^m\sharp\mathcal{R}^n}\frac{1}{|y|^{m-2}}\frac{t_i}{(t_i+d(x,y))^{n+1}}|b_{i}(y)|dy\\
&\leq C\sup_{z\in Q_i}\frac{1}{|z|^{m-2}}\frac{t_i}{(t_i+d(x,z))^{n+1}}\int_X |b_{i}(y)|dy\\
&\leq C\inf_{z\in Q_i}\frac{1}{|z|^{m-2}}\frac{t_i}{(t_i+d(x,z))^{n+1}}\int_X |b_{i}(y)|dy\\
&\leq  C\lambda\int_{\mathbb{R}^m\sharp\mathcal{R}^n} \frac{1}{|z|^{m-2}}\frac{t_i}{(t_i+d(x,z))^{n+1}}\chi_{Q_i}(z)dz.
\end{align*}


So for any $h\in L^2(\mathcal{R}^n\backslash K)$ with $\|h\|_{L^2(\mathcal{R}^n\backslash K)}=1$,
\begin{align*}
|\langle G_{2,i}, h\rangle|
&\leq C\lambda \int_{\mathbb{R}^m\backslash K}  \int_{\mathcal{R}^n\backslash K}\frac{1}{|z|^{m-2}}\frac{t_i}{(t_i+d(x,z))^{n+1}}|h(x)|dx\  \chi_{Q_i}(z)dz\\
&\leq C \lambda \int_{\mathbb{R}^m\backslash K}  \mathcal T(h)(z)  \chi_{Q_i}(z)dz,
\end{align*}
where the operator $\mathcal T$ is defined as 
$$\mathcal  T(h)(z) :=  \int_{\mathcal{R}^n\backslash K}   \frac{1}{|z|^{m-2}}\frac{t_i}{(t_i+d(x,z))^{n+1}}|h(x)|dx.$$
Once again, it is direct to see that $\mathcal T$ is a bounded operator on $L^2(\mathbb{R}^m\backslash K)$:
\begin{align*}
\|\mathcal  T(h)\|_{L^2(\mathbb{R}^m\backslash K)}^2
& \leq  \|h\|_{L^2(\mathcal{R}^n\backslash K)}^2 \int_{\mathbb{R}^m\backslash K}\int_{\mathcal{R}^n\backslash K}   {1\over |z|^{2m-4}}\frac{t^2_i}{(t_i+|x|+|z|)^{2(n+1)}} dx dz\\
& \leq \|h\|_{L^2(\mathcal{R}^n\backslash K)}^2\int_{\mathbb{R}^m\backslash K}\int_{\mathcal{R}^n\backslash K}   \frac{1}{t^n}\Big(\frac{t_i}{t_i+|x|}\Big)^{n+2} dx \frac{1}{|z|^{2m+n-4}} dz\\
&\leq C\|h\|_{L^2(\mathcal{R}^n\backslash K)}^2,
\end{align*}
where in the last inequality we use the condition that $m>n>2$.

As a consequence we obtain that
\begin{align*}
\Big\|\sum_{i\in \mathcal{I}_1} G_{2,i}\Big\|_{L^2(\mathcal{R}^n\backslash K)} &\leq C\lambda 
\Big\|\sum_{i\in \mathcal{I}_1} \chi_{Q_i}\Big\|_{L^2(\mathbb{R}^m\backslash K)} \\
&\leq C\lambda  \bigg(\sum_i |Q_i|\bigg)^{1/2}\\
&\leq  C\lambda^{1\over2}\|f_1\|_{L^1(\mathbb{R}^m\sharp\mathcal{R}^n)}^{1\over2}.
\end{align*}

Similar to the estimates for the terms $F_{2,i}$, we also obtain
\begin{align*}
\Big\|\sum_{i\in \mathcal{I}_1} G_{1,i}\Big\|_{L^2(\mathcal{R}^n\backslash K)}
&\leq  C\lambda^{1\over2}\|f_1\|_{L^1(\mathbb{R}^m\sharp\mathcal{R}^n)}^{1\over2}.
\end{align*}
Combining the estimates for $G_{1,i}$ and $G_{2,i}$ above, we obtain that \eqref{claim121case2} holds.

\small
We now verify \eqref{claim121case3}.
Consider the function $ e^{-t_i\sqrt{L}} b_{i}  (x)$ for $x\in  K$.
Since 
$$ e^{-t_i\sqrt{L}} b_{i}  (x) = \int_{\mathbb{R}^m\backslash K} P_{t_i}(x,y) b_{i} (y)dy, $$
applying the upper bound in point 5 in Theorem \ref{theorem poisson} for $ |P_{t_i}(x,y)| $ we obtain that 
\begin{align*}
|e^{-t_i\sqrt{L}} b_{i}  (x)|
 &\leq \int_{\mathbb R^m \sharp \mathcal R^n} |P_{t_i}(x,y)|\, |b_{i} (y)|dy\\
&\leq C\int_{\mathbb R^m \sharp \mathcal R^n} \frac{t_i}{(t_i+d(x,y))^{m+1}} |b_{i}(y)|dy +
C\int_{\mathbb R^m \sharp \mathcal R^n}  \frac{1}{|y|^{m-2}}\frac{t}{(t_i+d(x,y))^{n+1}}
 |b_{i}(y)|dy\\
&=: H_{1,i} + H_{2,i}.
\end{align*}
Arguing similarly to the estimates for the terms $G_{1,i}$ and $G_{2,i}$, we get the same estimates for the terms
$H_{1,i}$ and $H_{2,i}$, respectively.
This implies that 
\eqref{claim121case3} holds.

\bigskip

We now consider the term $I_{122}.$ Note that
\begin{align*}
I_{122}&\leq  \left|\left\{ x\in (\mathbb{R}^m\backslash K) \backslash \cup_i 8Q_i: \Big|\frak{M}(\sqrt L) \Big(\sum_j \big(I- e^{-t_j\sqrt{L}}\big) b_{j}\Big)  (x) \Big|>{\lambda\over12} \right\}\right|\\
&\leq {C\over \lambda} \sum_j \int_{(8Q_j)^c} |\frak{M}(\sqrt L) \big(I- e^{-t_j\sqrt{L}}\big) b_{j}(x) | dx.
\end{align*} 
Note that for each $j$, we get that 
%
%
\begin{align}\label{eI122}
&\int_{(8Q_j)^c} |\frak{M}(\sqrt L) \big(I- e^{-t_j\sqrt{L}}\big) b_{j}(x) | dx\\
&\leq \int_{(8Q_j)^c} \int_{Q_j} |k_{t_j}(x,y)| | b_{j}(y)|dy  dx\nonumber\\
&= \int_{Q_j}\ \int_{(8Q_j)^c}|k_{t_j}(x,y)| dx\ | b_{j}(y)|dy\nonumber
\end{align} 
where we use $k_{t_j}(x,y)$ to denote the kernel of of the operator $\frak{M}(\sqrt L) \big(I- e^{-t_j\sqrt{L}}\big) $.

By definition, we have
\begin{align*}
\frak{M}(\sqrt L) \big(I- e^{-t_j\sqrt{L}}\big) &= \int_0^\infty \sqrt L e^{-s\sqrt L} m(s)ds \int_0^{t_j} -{d\over dt} e^{-t\sqrt L} dt\\
&=\int_0^\infty \sqrt L e^{-s\sqrt L} m(s)ds \int_0^{t_j}  \sqrt L e^{-t\sqrt L} dt\\
&=\int_0^{t_j} \int_0^\infty (\sqrt L)^2  e^{-(s+t)\sqrt L}\,m(s)\, ds dt\\
&=\int_0^{t_j} \int_0^\infty  (s+t)^2(\sqrt L)^2  e^{-(s+t)\sqrt L}\,{m(s)\over (s+t)^2}\, ds dt.
\end{align*} 
Hence, we obtain that
\begin{align*}
k_{t_j}(x,y) =\int_0^{t_j} \int_0^\infty P_{s+t,2}(x,y) \,{m(s)\over (s+t)^2}\, ds dt.
\end{align*} 

We now claim that there exists an absolute positive constant $C$ such that
\begin{align}\label{eI122claim}
 \int_{(8Q_j)^c}|k_{t_j}(x,y)| dx\leq C.
\end{align} 

To see this, applying the kernel expression above and Case 5 in Theorem \ref{theorem poisson} for $P_{t,2}(x,y)$, we get that 
\begin{align*}
& \int_{(8Q_j)^c}|k_{t_j}(x,y)| dx\\
&\leq  \int_{(8Q_j)^c} \int_0^{t_j} \int_0^\infty |P_{s+t,2}(x,y)| \,{|m(s)|\over (s+t)^2}\, ds dt  dx\\
 &\leq  \int_{(8Q_j)^c} \int_0^{t_j} \int_0^\infty \frac{C}{(s+t)^m}\Big(\frac{s+t}{s+t+d(x,y)}\Big)^{m+2} \,{1\over (s+t)^2}\, ds dt  dx\\
 &\quad+  \int_{(8Q_j)^c} \int_0^{t_j} \int_0^\infty \frac{C}{(s+t)^n|x|^{m-2}|y|^{m-2}}\Big(\frac{s+t}{s+t+|x|+|y|}\Big)^{n+2} \,{1\over (s+t)^2}\, ds dt  dx\\
 &=: E_1+E_2.
\end{align*} 

We first consider the term $E_1$. Note that
\begin{align*}
 E_1&\leq  C\int_0^{t_j} \int_0^\infty \int_{d(x,y)\geq 2t_j}  \frac{1}{(s+t)^m}\Big(\frac{s+t}{s+t+d(x,y)}\Big)^{m+2} \,{1\over (s+t)^2}\, dx \,ds dt \\
 &\leq C\int_0^{t_j} \int_0^{t_j} \int_{d(x,y)\geq 2t_j}  \Big(\frac{1}{s+t+d(x,y)}\Big)^{m+2}\, dx \,ds dt \\
 &\quad+C  \int_0^{t_j} \int_{t_j}^\infty \int_{d(x,y)\geq 2t_j}  \frac{1}{(s+t)^m}\Big(\frac{s+t}{s+t+d(x,y)}\Big)^{m+2}\, dx \,{1\over (s+t)^2} \,ds dt \\
 &\leq C\int_0^{t_j} \int_0^{t_j} \int_{t_j}^\infty  \frac{1}{r^{m+2}} r^{m-1}\, dr \,ds dt +C\int_0^{t_j} \int_{t_j}^\infty  {1\over (s+t)^2} \,ds dt \\
 &\leq C,
\end{align*} 
where in the last inequality, we use polar coordinates to estimate the first term and we use the following fact for the second term
$$\int_{d(x,y)\geq 2t_j}  \frac{1}{(s+t)^m}\Big(\frac{s+t}{s+t+d(x,y)}\Big)^{m+2}\, dx\leq C.$$

Next we consider the term $E_2$. Note that
\begin{align*}
E_2&\leq C\int_0^{t_j} \int_0^{t_j}   \int_{d(x,y)\geq 2t_j}\frac{1}{|x|^{m-2}|y|^{m-2}}\Big(\frac{1}{|x|+|y|}\Big)^{n+2}\,dx \,\, ds dt \\
&\quad +C \int_0^{t_j} \int_{t_j}^\infty   \int_{d(x,y)\geq 2t_j}\frac{1}{(s+t)^n|x|^{m-2}|y|^{m-2}}\Big(\frac{s+t}{s+t+|x|+|y|}\Big)^{n+2}\,dx \,{1\over (s+t)^2}\, ds dt \\
&=:E_{21}+E_{22}.
\end{align*} 
As for $E_{21}$, we first suppose $t_j\geq1$.
Then by
noting that $d(x,y)\leq |x|+|y|$ and that $|x|\geq1$ and $|y|\geq 1$, we have
$$ {1\over |x|^{m-2}} \leq {1\over d(x,y)^{m-2}} \quad {\rm or}\quad {1\over |y|^{m-2}} \leq {1\over d(x,y)^{m-2}},$$
which implies that
$$ \frac{1}{|x|^{m-2}|y|^{m-2}} \leq {1\over d(x,y)^{m-2}}. $$
As a consequence,
\begin{align*}
E_{21}&\leq C\int_0^{t_j} \int_0^{t_j}   \int_{d(x,y)\geq 2t_j}\frac{1}{d(x,y)^{m-2}}\Big(\frac{1}{d(x,y)}\Big)^{n+2}\,dx \,\, ds dt \\
& \leq C\int_0^{t_j} \int_0^{t_j}\int_{t_j}^\infty {1\over r^{m+n}} r^{m-1}\, drdsdt \leq C {t_j^2\over t_j^n}\leq C.
\end{align*} 
We now suppose $t_j<1$. Then it is direct that
\begin{align*}
E_{21}&\leq C\int_0^{1} \int_0^{1}   \int_{d(x,y)\geq 2t_j} \frac{1}{|x|^{m-2}|y|^{m-2}}\Big(\frac{1}{|x|+|y|}\Big)^{n+2}\,dx \,\, ds dt \\
& \leq C \int_{\mathbb{R}^m\backslash K} \frac{1}{|x|^{m-2}}\Big(\frac{1}{|x|}\Big)^{n+2}\,dx \leq C.
\end{align*} 
As for $E_{22}$, again, noting that $d(x,y)\leq |x|+|y|$, we have
\begin{align*}
E_{22}&\leq C \int_0^{t_j} \int_{t_j}^\infty   \int_{d(x,y)\geq 2t_j}\frac{1}{(s+t)^n|x|^{m-2}|y|^{m-2}}\Big(\frac{s+t}{s+t+|x|+|y|}\Big)^{n+2}\,dx \,{1\over (s+t)^2}\, ds dt \\
&\leq C \int_0^{t_j} \int_{t_j}^\infty   \int_{d(x,y)\geq 2t_j}\frac{1}{(s+t)^n|x|^{m-2}|y|^{m-2}}\Big(\frac{s+t}{s+t+|x|+|y|}\Big)^{n}\,dx \,{1\over (s+t)^2}\, ds dt \\
&\leq C \int_0^{t_j} \int_{t_j}^\infty   \int_{d(x,y)\geq 2t_j}\frac{1}{|x|^{m-2}|y|^{m-2}}\Big(\frac{1}{|x|+|y|}\Big)^{n}\,dx \,{1\over (s+t)^2}\, ds dt \\
&\leq C \int_0^{t_j} \int_{t_j}^\infty   \int_{\mathbb R^m\backslash K}\frac{1}{|x|^{m-2}}\Big(\frac{1}{|x|}\Big)^{n}\,dx \,{1\over (s+t)^2}\, ds dt \\
&\leq C \int_0^{t_j} \int_{t_j}^\infty    \,{1\over (s+t)^2}\, ds dt \\
&\leq C.
\end{align*} 

Combining the estimates of $E_{22}$, $E_{12}$ and $E_1$, we obtain that the claim \eqref{eI122claim} holds.
As a consequence, from \eqref{eI122} we obtain that for each $j$,
\begin{align*} 
\int_{(8Q_j)^c} |\frak{M}(\sqrt L) \big(I- e^{-t_j\sqrt{L}}\big) b_{j}(x) | dx
\leq \int_{Q_j}\ | b_{j}(y)|dy\leq C\lambda |Q_j|,
\end{align*} 
which implies that
\begin{align*}
I_{122}
&\leq {C\over \lambda} \sum_j \int_{(8Q_j)^c} |\frak{M}(\sqrt L) \big(I- e^{-t_j\sqrt{L}}\big) b_{j}(x) | dx\\
&\leq C\sum_j|Q_j| \leq {C\over \lambda} \|f_1\|_{L^1(\mathbb{R}^m\sharp\mathcal{R}^n)}.
\end{align*} 

We now consider the term $I_{123}$. Note that for each $i\in \mathcal{I}_2$ we have $t_i\geq 1/2$.   Fix $i\in \mathcal{I}_2$. Denote by $k_{\frak{M}(\sqrt{L})}(x,y)$
 the associated kernel of $\frak{M}(\sqrt{L})$. For $x\in (\mathbb{R}^m\backslash K)\backslash 8Q_i$ and $y\in Q_i$, by point 5 in Theorem \ref{theorem poisson} we have
$$
\begin{aligned}
|k_{\frak{M}(\sqrt{L})}(x,y)|&\leq \int_0^\infty |P_{t}(x,y)|\frac{dt}{t}\\
&\leq \int_0^\infty\frac{C}{t^{m+1}}\Big(\frac{t}{t+d(x,y)}\Big)^{m+1}dt+ \int_0^\infty\frac{C}{t^{n+1}|x|^{m-2}|y|^{m-2}}\Big(\frac{t}{t+|x|+|y|}\Big)^{n+1}dt\\
&=:K_1(x,y)+K_2(x,y).
\end{aligned}
$$
Since $d(x,y)\sim d(x,x_{Q_i})$, we have
$$
\begin{aligned}
K_1(x,y)&\leq \int_0^\infty\frac{C}{t^{m+1}}\Big(\frac{t}{t+d(x,x_{Q_i})}\Big)^{m+1}dt\\
&\leq \int_0^{d(x,x_{Q_i})}\frac{C}{d(x,x_{Q_i})^{m+1}}dt+\int^\infty_{d(x,x_{Q_i})}\frac{C}{t^{m+1}}dt\\
&\leq \frac{C}{d(x,x_{Q_i})^{m}}.
\end{aligned}
$$
Using the fact that $|x||y|\gtrsim |x|+|y|\gtrsim d(x,y)\gtrsim d(x,x_{Q_i})$ we have
$$
\begin{aligned}
K_2(x,y)&\leq \int_0^\infty\frac{C}{t^{n+1}d(x,x_{Q_i})^{m-2}}\Big(\frac{t}{t+d(x,x_{Q_i})}\Big)^{n+1}dt\\
&\leq \int_0^{d(x,x_{Q_i})}\frac{C}{d(x,x_{Q_i})^{m+n-1}}dt+\int^\infty_{d(x,x_{Q_i})}\frac{C}{t^{n+1}d(x,x_{Q_i})^{m-2}}dt\\
&\leq \frac{C}{d(x,x_{Q_i})^{m+n-2}}\\
&\leq \frac{C}{d(x,x_{Q_i})^{m}}
\end{aligned}
$$
where in the last inequality we used $d(x,x_{Q_i})\geq 2t_i\geq 1$.

From the estimates of $K_1(x,y)$ and $K_2(x,y)$, for each $i\in \mathcal{I}_2$ and $x\in (\mathbb{R}^m\backslash K)\backslash 8Q_i$ we have
\[
\sup_{y\in Q_i}|k_{\frak{M}(\sqrt{L})}(x,y)|\leq \frac{C}{d(x,x_{Q_i})^{m}}.
\]
Moreover, observe that since $i\in \mathcal{I}_2$ and $x\in (\mathbb{R}^m\backslash K)\backslash 8Q_i$ we have
$$
\frac{1}{d(x,x_{Q_i})^{m}}\sim \frac{1}{|x|^{m}}.
$$
As a consequence,
for each $i\in \mathcal{I}_2$ and $x\in (\mathbb{R}^m\backslash K)\backslash 8Q_i$ we have
\[
\sup_{y\in Q_i}|k_{\frak{M}(\sqrt{L})}(x,y)|\leq \frac{C}{|x|^{m}}.
\]
This implies that for each $x\in (\mathbb{R}^m\backslash K) \backslash \cup_i 8Q_i$, we have
\[
\big|\sum_{i\in \mathcal{I}_2} \frak{M}(\sqrt L) b_{i}  (x) \big|\le C\frac{\sum_{i\in \mathcal{I}_2}\|b_i\|_{L^1(\mathbb{R}^m\sharp\mathcal{R}^n)}}{|x|^m}.
\]
Therefore,
\begin{align*}
	I_{123}&\leq \left|\left\{ x\in (\mathbb{R}^m\backslash K) \backslash \cup_i 8Q_i: \big|\sum_{i\in \mathcal{I}_2} \frak{M}(\sqrt L) b_{i}  (x) \big|>{\lambda\over18} \right\}\right|\\
	&\leq \left|\left\{ x\in (\mathbb{R}^m\backslash K) \backslash \cup_i 8Q_i: C|x|^{-m}\left(\sum_{i\in \mathcal{I}_2}\|b_i\|_{L^1(\mathbb{R}^m\sharp\mathcal{R}^n)}\right)>{\lambda\over18} \right\}\right|\\
	&\leq C\frac{\sum_{i\in \mathcal{I}_2}\|b_i\|_{L^1(\mathbb{R}^m\sharp\mathcal{R}^n)}}{\lambda}\\
	&\leq C\frac{\|f\|_{L^1(\mathbb{R}^m\sharp\mathcal{R}^n)}}{\lambda}.
\end{align*} 
Combining all cases of $I_{11}$, $I_{13}$, $I_{121}$, $I_{122}$ and $I_{123}$, we obtain that 
\begin{align*}
I_{1}\leq {C\over \lambda} \|f_1\|_{L^1(X)}.
\end{align*}

\subsection{Estimate of $I_2$}

We now consider the term $I_2$. Note that in the case, $x$ is in the large end $\mathbb{R}^m\backslash K$ and
the function $f_2$ is supported in the small end $\mathcal{R}^n\backslash K$, and hence the measure will become 
non-doubling since if we enlarge a ball contained in  $\mathcal{R}^n\backslash K$, then the enlargement can be
partially contained in $\mathbb{R}^m\backslash K$. The standard Calder\'on--Zygmund decomposition on non-homogeneous space such as in \cite{NTV2, To2} does not apply since in that decomposition, we only know the existence of a sequence of  Calder\'on--Zygmund cubes but we do not know where they are exactly. And the Poisson kernel upper bound depends heavily on the position of the variables $x$ and $y$ in different ends.

Thus, to deal with this case, we use a Whitney type decomposition of the level set $\Omega$ below and then we 
make clever use of the Poisson kernel upper bound in this case to handle the weak type estimate, without enlarging those cubes, which avoids the case of non-doubling measure.  The genesis of this approach is an adaptation of an idea from \cite{NTV2}.

Note that $f_2$ is supported in $\mathcal{R}^n\backslash K$. We now split $\mathcal{R}^n\backslash K$ into two parts according to
$f_2$. 
Define
$$ F:= \{x\in \mathcal{R}^n\backslash K: M_2(f_2)\leq \lambda\}$$
and 
$$\Omega:= \{x\in \mathcal{R}^n\backslash K: M_2(f_2)> \lambda\}, $$
where $M_2$ is the Hardy--Littlewood maximal function defined on $\mathcal{R}^n\backslash K$.

Then we define
$$ f_{2,\lambda}(x):=f_2(x) \chi_{ F}(x)\quad{\rm and}\quad  f_{2}^{\lambda}(x):=f_2(x) \chi_{ \Omega}(x).$$
Then we have
\begin{align*}
I_2&\leq \left|\left\{ x\in \mathbb{R}^m\backslash K: |\frak{M}(\sqrt L)f_{2,\lambda} (x) |>{\lambda\over6} \right\}\right|\\
&\quad +  \left|\left\{ x\in \mathbb{R}^m\backslash K: |\frak{M}(\sqrt L)f_2^\lambda (x) |>{\lambda\over6} \right\}\right|\\
&=: I_{21}+I_{22}.
\end{align*} 

As for $I_{21}$, by using the $L^2$ boundedness of $m(\sqrt L)$, we obtain that
\begin{align*}
I_{21}&= |\{ x\in \mathbb{R}^m\backslash K: |\frak{M}(\sqrt L)f_{2,\lambda} (x) |>{\lambda\over6} \}|\\
&\leq {C\over \lambda^2} \|f_{2,\lambda}\|_{L^2(\mathcal{R}^n\backslash K)}^2\\
&\leq {C\over \lambda} \|f\|_{L^1(\mathbb{R}^m\sharp\mathcal{R}^n)},
\end{align*} 
where we use the fact that $|f_{2,\lambda}(x)|=|f_2(x)| \chi_{ F}(x)\leq |M_2(f_2)(x)| \chi_{ F}(x) \leq \lambda$.  

As for $I_{22}$, we consider the function $f_2^\lambda$. We now apply a covering lemma in \cite{CW} (see also \cite[Lemma 5.5]{DKP}) for the set $\Omega$ in the homogeneous space $\mathcal{R}^n$ to obtain a collection of balls $\{Q_{i}:=B(x_i,r_{i}): x_{i}\in \Omega, r_i=d(x_i,\Omega^c)/2, i=1,\ldots\}$ so that
\begin{enumerate}[{\rm (i)}]
	\item $\displaystyle \Omega=\cup_i Q_i$;
	\item $\displaystyle  \{B(x_i,r_i/5)\}_{i=1}^\infty$ are disjoint;
	\item there exists a universal constant $C$ so that $\sum_k \chi_{Q_k}(x)\le C$ for all $x\in \Omega$.
\end{enumerate}
Hence, we can further decompose
$$ f_2^\lambda(x) =  \sum_i f_{2,i}^\lambda(x),$$
where $f_{2,i}^\lambda(x) = \frac{\chi_{Q_i}(x)}{\sum_k \chi_{Q_k}(x)}f_{2}^\lambda(x)$.

Next, note that for $x\in \mathbb{R}^m\backslash K$,
\begin{align*}
|\frak{M}(\sqrt L)(f_{2,i}^\lambda) (x)| & = \bigg|\int_0^\infty t\sqrt L \exp(-t\sqrt L)(f_{2,i}^\lambda)(x)  \tilde{m}(t) {dt\over t}\,  \bigg|\\
&\leq \int_0^\infty \int_{Q_i} |P_{t,1}(x,y)| \, |\tilde{m}(t)| \, |f_{2,i}^\lambda(y)|\, dy\frac{dt}{t}.
\end{align*} 
Applying the upper bound in point 4 in Theorem \ref{theorem poisson} for $ |P_{t,1}(x,y)| $ we obtain that
\begin{equation}\label{eq- E123}
\begin{aligned}
|\frak{M}(\sqrt L)(f_{2,i}^\lambda) (x)| &\leq C\int_0^\infty \int_{Q_i} \bigg(\frac{t}{(t+d(x,y))^{m+1}}+ \frac{1}{|x|^{m-2}}\frac{t}{(t+d(x,y))^{n+1}}\\
&\hskip3cm+\frac{1}{|y|^{n-2}}\frac{t}{(t+d(x,y))^{m+ 1}}\bigg)\, |f_{2,i}^\lambda(y)|\, dy\frac{dt}{t}.
\end{aligned} 
\end{equation}
Note that $d(x,y)\approx |x|+|y|$ since $x\in \mathbb{R}^m\backslash K$ and  $y\in \mathcal{R}^n\backslash K$. Hence,
\[
\begin{aligned}
\int_0^\infty\bigg(\frac{t}{(t+d(x,y))^{m+1}}&+ \frac{1}{|x|^{m-2}}\frac{t}{(t+d(x,y))^{n+1}}+\frac{1}{|y|^{n-2}}\frac{t}{(t+d(x,y))^{m+ 1}}\bigg)\frac{dt}{t}\\
&\le C\int_0^\infty\bigg(\frac{t}{(t+|x|)^{m+1}}+ \frac{1}{|x|^{m-2}}\frac{t}{(t+|x|)^{n+1}}\bigg)\frac{dt}{t}.
\end{aligned}
\]
Splitting this into two integrals over $(0,|x|)$ and $(|x|,\infty)$, by a straightforward calculation we have
\begin{equation}\label{eq-I2}
\begin{aligned}
\int_0^\infty\bigg(\frac{t}{(t+d(x,y))^{m+1}}+ \frac{1}{|x|^{m-2}}\frac{t}{(t+d(x,y))^{n+1}}+\frac{1}{|y|^{n-2}}\frac{t}{(t+d(x,y))^{m+ 1}}\bigg)\frac{dt}{t}\le \frac{C}{|x|^m}.
\end{aligned}
\end{equation}
Inserting into \eqref{eq- E123}, we have  
\begin{align*}
|\frak{M}(\sqrt L)(f_{2,i}^\lambda) (x)| &\leq  {C\over |x|^m} \ \int_{Q_i}  |f_{2,i}^\lambda(y)| dy,
\end{align*} 
which implies that
\begin{align*}
|\frak{M}(\sqrt L)(f_{2}^\lambda) (x)| &\leq  \sum_i |\frak{M}(\sqrt L)(f_{2,i}^\lambda) (x)|\\
&\leq  C{1\over |x|^m} \sum_i \int_{Q_i}  |f_{2,i}^\lambda(y)| dy\\
&\leq C{1\over |x|^m} \|f_2\|_{L^1(\mathbb{R}^m\sharp\mathcal{R}^n)}.
\end{align*} 
Hence, we obtain that
\begin{align*}
I_{22}&=  \left|\left\{ x\in \mathbb{R}^m\backslash K: |\frak{M}(\sqrt L)f_2^\lambda (x) |>{\lambda\over6} \right\}\right|\\
&\leq  \left|\left\{ x\in \mathbb{R}^m\backslash K: C{1\over |x|^m} \|f_2\|_{L^1(\mathbb{R}^m\sharp\mathcal{R}^n)}>{\lambda\over6} \right\}\right|\\
&=  \left|\left\{ x\in \mathbb{R}^m\backslash K:\ |x|^m< {6C\|f_2\|_{L^1(\mathbb{R}^m\sharp\mathcal{R}^n)} \over \lambda} \right\}\right|\\
&\leq {C\|f_2\|_{L^1(\mathbb{R}^m\sharp\mathcal{R}^n)} \over \lambda}\\
&\leq {C\|f\|_{L^1(\mathbb{R}^m\sharp\mathcal{R}^n)} \over \lambda}.
\end{align*}

\subsection{Estimate of $I_3$}

For the term $I_3$, we point out that we can handle this case by using the same approach as in the estimates for the term $I_1$ with minor modifications and hence we omit the details.

\subsection{Estimate of $II_1$}

For the term $II_1$, we point out that we can handle this case by using similar way as in the estimates for the term $I_2$.
We sketch the proof as follows.

Define
$ F:= \{x\in \mathbb{R}^m\backslash K: M_1(f_1)\leq \lambda\}$
and 
$\Omega:= \{x\in \mathbb{R}^m\backslash  K: M_1(f_1)> \lambda\}, $
where $M_1$ is the Hardy--Littlewood maximal function defined on $\mathbb{R}^m\backslash K$.
Then let
$ f_{1,\lambda}(x):=f_1(x) \chi_{ F}(x)$ and $f_{1}^{\lambda}(x):=f_1(x) \chi_{ \Omega}(x).$

Then we have
\begin{align*}
II_1&\leq \left|\left\{ x\in \mathcal{R}^n\backslash K: |\frak{M}(\sqrt L)f_{1,\lambda} (x) |>{\lambda\over6} \right\}\right| +  \left|\left\{ x\in \mathcal{R}^n\backslash K: |\frak{M}(\sqrt L)f_1^\lambda (x) |>{\lambda\over6} \right\}\right|\\
&=: II_{11}+II_{12}.
\end{align*} 
Using the $L^2$ boundedness of $\frak{M}(\sqrt L)$ and the fact that $|f_{1,\lambda}(x)|\leq \lambda$, we obtain 
$
II_{11}\leq {C\over \lambda} \|f\|_{L^1(\mathbb{R}^m\sharp\mathcal{R}^n)}.
$

As for $II_{12}$, by using the Whitney decomposition, we obtain 
$ \Omega=\bigcup_i Q_i $
such that
$\sum_i|Q_i| =|\Omega|$, which gives
$$ f_1^\lambda(x) =  \sum_i f_{1,i}^\lambda(x),$$
where $f_{1,i}^\lambda(x) = f_{1}^\lambda(x)\chi_{Q_i}(x)$.

Next, for $x\in \mathcal{R}^n\backslash K$,
\begin{align*}
|\frak{M}(\sqrt L)(f_{1,i}^\lambda) (x)| 
\leq \int_0^\infty \int_{Q_i} |P_{t,1}(x,y)| \, |\tilde{m}(t)| \, |f_{1,i}^\lambda(y)|\, dy\frac{dt}{t}.
\end{align*} 
Applying the upper bound in point 4 in Theorem \ref{theorem poisson} for $ |P_{t,1}(x,y)| $ we obtain that
\begin{align*}
|\frak{M}(\sqrt L)(f_{1,i}^\lambda) (x)| &\leq C\int_0^\infty \int_{Q_i} \bigg(\frac{t}{(t+d(x,y))^{m+1}}+ \frac{1}{|y|^{m-2}}\frac{t}{(t+d(x,y))^{n+1}}\\
&\hskip3cm+\frac{1}{|x|^{n-2}}\frac{t}{(t+d(x,y))^{m+ 1}}\bigg)\, |f_{1,i}^\lambda(y)|\, dy\frac{dt}{t}\\
&\leq  {C\over |x|^n} \ \int_{Q_i}  |f_{1,i}^\lambda(y)| dy,
\end{align*} 
where the last inequality follows from similar estimates for \eqref{eq-I2}. This implies that
\begin{align*}
|\frak{M}(\sqrt L)(f_{1}^\lambda) (x)| &\leq  \sum_i |m(\sqrt L)(f_{1,i}^\lambda) (x)|
\leq  C{1\over |x|^n} \sum_i \int_{Q_i}  |f_{1,i}^\lambda(y)| dy
\leq C{1\over |x|^n} \|f_1\|_{L^1(\mathbb{R}^m\sharp\mathcal{R}^n)}.
\end{align*} 
Hence, we obtain that
\begin{align*}
II_{12}&=  \left|\left\{ x\in \mathcal{R}^n\backslash K: |\frak{M}(\sqrt L)f_1^\lambda (x) |>{\lambda\over6} \right\}\right|\leq  \left|\left\{ x\in \mathcal{R}^n\backslash K: C{1\over |x|^n} \|f_1\|_{L^1(\mathbb{R}^m\sharp\mathcal{R}^n)}>{\lambda\over6} \right\}\right|\\
&=  \left|\left\{ x\in \mathcal{R}^n\backslash K:\ |x|^n< {6C\|f_1\|_{L^1(\mathbb{R}^m\sharp\mathcal{R}^n)} \over \lambda} \right\}\right|\\
&\leq {C\|f\|_{L^1(\mathbb{R}^m\sharp\mathcal{R}^n)} \over \lambda}.
\end{align*}

\subsection{Estimate of $II_2$}

We will apply a similar approach as that in \cite{DMc} and using similar estimates for the term $I_1$ in our Section 3.1 to estimate $II_2$.

We restrict the setting to $\mathcal{R}^n\backslash K$, and $f_2 $ is in $L^1(\mathcal{R}^n\backslash K)$.
We now extend $f_2$ to all of $\mathbb{R}^n$ by zero extension, i.e., define $f_2(x):=0$ when $x\in K$.

Similar to the Calder\'on--Zygmund decomposition in $I_1$, we get $$f_2 (x)= g(x) + \sum_jb_j(x)$$
with $\|g\|_{L^1(\mathbb{R}^m\sharp\mathcal{R}^n)}\leq \|f\|_{L^1(\mathbb{R}^m\sharp\mathcal{R}^n)}$ and $\|g\|_{L^\infty(\mathbb{R}^m\sharp\mathcal{R}^n)}\leq C\lambda$ and 
$$ b_j(x):=\bigg( f_2(x)- {1\over |Q_j|}\int_{Q_j}f_2(y) dy \bigg)\chi_{Q_j}(x).$$
Then we get
\begin{align*}
II_2&\leq \left|\left\{ x\in \mathcal{R}^n\backslash K: |\frak{M}(\sqrt L )g (x) |> {\lambda\over6} \right\}\right|\\
&\quad+  \left|\left\{ x\in (\mathcal{R}^n\backslash K) \backslash \cup_i 8Q_i: \big|\frak{M}(\sqrt L ) \big(\sum_i b_{i}\big)  (x) \big|>{\lambda\over6} \right\}\right|+ | \cup_i 8Q_i|\\
&=: II_{21}+II_{22}+II_{23}. 
\end{align*} 
By using the $L^2$ boundedness of $\frak{M}(\sqrt L )$ and the fact that $\|g\|_{L^\infty(X)}\leq C\lambda$, we obtain that
$II_{21}\leq {C\over \lambda} \|f\|_{L^1(\mathbb{R}^m\sharp\mathcal{R}^n)}.
$
Next, from the doubling condition in this case, we get that
$
II_{23}\leq  C \sum_i |Q_i|\leq  C{\|f\|_{L^1(\mathbb{R}^m\sharp\mathcal{R}^n)}\over \lambda}.
$
For the term $II_{22}$, we have
\begin{align*}
II_{22}&\leq  \left|\left\{ x\in (\mathcal{R}^n\backslash K) \backslash \cup_i 8Q_i: \big|\frak{M}(\sqrt L ) \big(\sum_{i} e^{-t_i\sqrt{ L }} b_{i}\big)  (x) \big|>{\lambda\over18} \right\}\right|\\
&\quad+ \left|\left\{ x\in (\mathcal{R}^n\backslash K) \backslash \cup_i 8Q_i: \Big|\frak{M}(\sqrt L ) \Big(\sum_{i} \big(I- e^{-t_i\sqrt{ L }}\big) b_{i}\Big)  (x) \Big|>{\lambda\over18} \right\}\right|\\
&=: II_{221}+II_{222},
\end{align*} 
where for each $i$, $t_i$ is the side length of the cube $Q_i$.
Note that the term $II_{222}$ can be handled similarly by  using the same approach as that for $I_{122}$ and using upper bound in point 6 in Theorem \ref{theorem poisson} for $ |P_{t,2}(x,y)|$, which yields that 
$II_{222}$  is bounded by $C{\|f\|_{L^1(\mathbb{R}^m\sharp\mathcal{R}^n)}\over \lambda}$.

As for $II_{221}$,  we now split all the $Q_i$'s into two groups: 
$$\mathcal{J}_1:=\{i: \text{none of the corners of  $Q_i$ is the origin}\},$$ 
and 
$$\mathcal{J}_2=\{i: \text{one of the corners of  $Q_i$ is the origin}\}.
$$
Similarly $I_{12}$, we need only to claim that
\begin{align}\label{claim221}
\Big\|\sum_{i\in \mathcal J_1} e^{-t_i\sqrt{L}} b_{i}\Big\|_{L^2(\mathbb{R}^m\sharp\mathcal{R}^n)}
\leq  C\lambda^{1\over2}\|f_2\|_{L^1(\mathbb{R}^m\sharp\mathcal{R}^n)}^{1\over2}.
\end{align}
To see this claim, it suffices to show the following 3 cases:
\begin{align}\label{claim221case1}
\Big\|\sum_{i\in \mathcal J_1} e^{-t_i\sqrt{L}} b_{i}\Big\|_{L^2(\mathbb{R}^m\backslash K)}
\leq  C\lambda^{1\over2}\|f_2\|_{L^1(\mathbb{R}^m\sharp\mathcal{R}^n)}^{1\over2},
\end{align}
\begin{align}\label{claim221case2}
\Big\|\sum_{i\in \mathcal J_1} e^{-t_i\sqrt{L}} b_{i}\Big\|_{L^2(\mathcal{R}^n\backslash K)}
\leq  C\lambda^{1\over2}\|f_2\|_{L^1(\mathbb{R}^m\sharp\mathcal{R}^n)}^{1\over2},
\end{align}
and
\begin{align}\label{claim221case3}
\Big\|\sum_{i\in \mathcal J_1} e^{-t_i\sqrt{L}} b_{i}\Big\|_{L^2(K)}
\leq  C\lambda^{1\over2}\|f_2\|_{L^1(\mathbb{R}^m\sharp\mathcal{R}^n)}^{1\over2}.
\end{align}
We now point out that \eqref{claim221case1} can be obtained by using similar estimates as those for \eqref{claim121case2} and that \eqref{claim221case3} can be obtained by using similar estimates as those for \eqref{claim121case3}.  We omit the details.

As for \eqref{claim221case2}, 
applying the upper bound in point 6 in Theorem \ref{theorem poisson} for $ |P_{t_i}(x,y)| $ we obtain that
\begin{align*}
|e^{-t_i\sqrt{ L}} b_{i}  (x)| &\leq \int_{\mathcal{R}^n\backslash K} |P_{t_i}(x,y)|\, |b_{i} (y)|dy\\
&\leq C\int_{\mathcal{R}^n\backslash K} \bigg({t_i\over (t_i+d(x,y))^{m+1}}+{t_i\over  (t_i+d(x,y))^{n+1}}\bigg)\, |b_{i}(y)|dy\\
&=: \mathcal F_{1,i} + \mathcal F_{2,i}.
\end{align*}
For the term $\mathcal F_{2,i}$, by using smilar technique of the sup--inf estimate as in the estimate for $F_{2,i}$
in Subsection 3.1, we obtain that
$ \big\langle \sum_{i}  \mathcal F_{2,i}, h\big\rangle \leq C\lambda \big\langle  M_2(h),  \sum_{i}  \chi_{Q_i}\big\rangle $ for any $h$ with $\|h\|_{L^2(\mathcal{R}^n\backslash K)}=1$,
which yields that 
\begin{align*}
\Big\|\sum_{i\in \mathcal J_1} F_{2,i}\Big\|_{L^2(\mathbb{R}^m\backslash K)} \leq  C\lambda^{1\over2} \|f_2\|_{L^1(\mathbb{R}^m\sharp\mathcal{R}^n)}^{1\over2}.
\end{align*}
For the term $\mathcal F_{2,i}$, we consider the position of $Q_i$, the support of $b_i$, as follows: if one of the corners of $Q_i$ is origin, then $t_i\geq {1\over2}$, since otherwise the function $f_2$ on $Q_i$ is zero which yields that this $Q_i$ can not be chosen from the Calder\'on--Zygmund decomposition; if none of the corners of
$Q_i$ is origin, then if $t_i<1$, $d(Q_i,0)\geq {1\over2}$.
Combining all these cases, we get that
$${t_i\over (t_i+d(x,y))^{m+1}}\leq C{t_i\over (t_i+d(x,y))^{n+1}},$$
which shows that 
\begin{align*}
\Big\|\sum_{i\in \mathcal J_1} F_{1,i}\Big\|_{L^2(\mathbb{R}^m\backslash K)} \leq C \Big\|\sum_{i\in \mathcal I_1} F_{2,i}\Big\|_{L^2(\mathbb{R}^m\backslash K)} \leq  C\lambda^{1\over2} \|f_2\|_{L^1(\mathbb{R}^m\sharp\mathcal{R}^n)}^{1\over2}.
\end{align*}

\subsection{Estimate of $II_3$, $III_1$, $III_2$, $III_3$}
We point out that the estimates of $II_3$ follows from the upper bound in point 3 in Theorem \ref{theorem poisson} for $ |P_{t,1}(x,y)| $ and from similar estimates as for $I_3$
in Subsection 3.3. The estimates of $III_1$ and $III_2$ can be obtained by using similar techniques as in $I_3$ and $II_3$, respectively.  $III_3$ can also be obtained using similar approaches as in $II_3$. We omit the details here.

\bigskip
{\bf Acknowledgement:}  T. A. Bui, X. T. Duong and J. Li are supported by the Australian Research Council through the 
research grant ARC DP 160100153.  B. D. Wick supported in part by National Science Foundation grant DMS \# 1560955.
The authors would like to thank the referee for suggestions to improve the presentation of the paper.


\begin{thebibliography}{99999}

\bibitem{AM} P. Auscher, and J. M. Martell, Weighted norm inequalities, off-diagonal estimates and elliptic operators.
 Part I: general operator theory and weights, {Adv. Math.} \textbf{212} (2007), 225--276.


\bibitem{BD1} T. A.  Bui and X. T. Duong, Hardy spaces, {Regularized BMO spaces and the boundedness of Calder\'on-Zygmund operators on non-homogeneous spaces},  J Geom Anal {\bf 23} (2013), 895--932.


\bibitem{Car} G. Carron, Riesz transforms on connected sums, Ann. Inst. Fourier (Grenoble),  {\bf57} (2007), 2329--2343. 

\bibitem{CD1}  T. Coulhon and X. T. Duong, Riesz transforms for $1\leq p\leq
2$, {Trans. Amer. Math. Soc.} \textbf{351} (1999),  1151--1169.



\bibitem{CSY} P. Chen, A. Sikora and L. Yan, Spectral multipliers via resolvent type estimates on non-homogeneous metric measure spaces,  arXiv:1609.01871.

\bibitem{CW} R.R. Coifman and G. Weiss, Extensions of Hardy spaces and their use in analysis, Bull. Amer. Math. Soc. 83 (1977), 569-645.

\bibitem{Cowling} M. Cowling, Harmonic Analysis on semigroups, Annals of Mathematics, Vol 117, No2, (1983), 267--283.


\bibitem{DKP} S. Dekel, G. Kerkyacharian, G. Kyriazis and P. Petrushev, Hardy spaces associated with non-negative self-adjoint operators,  Studia Math. 239 (2017), no. 1, 17--54.

\bibitem{DHY} D. Deng, Y. Han and D. Yang, Besov spaces with non--doubling measures, Trans. Amer. Math. Soc. {\bf 358} (2006), 2965--3001.
 
\bibitem{DMc} X. T. Duong and A. McIntosh,  Singular integral operators with non-smooth kernels on irregular domains. Rev. Mat. Iberoamericana {\bf15} (1999), no. 2, 233--265. 

\bibitem{DLS} X. T. Duong, J. Li and A. Sikora,
{\it Boundedness of maximal functions on non-doubling manifolds with ends}, Proceedings of the Centre for Mathematics and its Applications, Australia, {\bf 45}, (2012), 37--47.


\bibitem {DY1} X. T. Duong and L. Yan, {New function spaces of BMO type,
John-Nirenberg inequality, interpolation and applications}, 
    Comm. Pure Appl. Math. {\bf LVIII} (2005), 1375--1420.

 \bibitem{DY2} X. T. Duong and L. Yan, Duality of Hardy and BMO spaces
associated with operators with heat kernel bounds, {J. Amer. Math. Soc.}
\textbf{18} (2005), 943--973.



\bibitem{GS} A. Grigor'yan and L. Saloff-Coste, {\it  Heat kernel on manifolds with ends }, Ann. Inst. Fourier (Grenoble), no.5, {\bf59} (2009), 1917--1997.

\bibitem{HY} Y. Han and D. Yang, Triebel--Lizorkin spaces with non--doubling measures, Studia Math. {\bf 162} (2004), 105--140.

\bibitem{HS} A. Hassell and A. Sikora, Riesz transforms on a class of non-doubling manifolds, arXiv:1805.00132.
 
\bibitem{HLMMY} S. Hofmann, G. Lu, D. Mitrea, M. Mitrea and L. Yan, {\it Hardy spaces associated to non-negative self-adjoint operators satisfying Davies-Gaffney estimates}, {Mem. Amer. Math. Soc.} \textbf{214} (2011).

\bibitem{Hy} T. Hyt\"onen, A framework for non-homogeneous analysis on metric spaces, and the RBMO space of Tolsa, Publ. Mat. {\bf 54} (2010) no. 2, 485--504.

 \bibitem{HM}  T. Hyt\"onen and H. Martikainen, Non-homogeneous $Tb$ theorem and random dyadic cubes on metric
measure spaces, Journal of Geometric Analysis {\bf 22} (2012), 1071--1107.

 \bibitem{HYY}  T. Hyt\"onen, Da.  Yang, and Do. Yang,  The Hardy space $H^1$ on non-homogeneous metric spaces,  Mathematical Proceedings of the Cambridge Philosophical Society {\bf 153} (2012), 9--31.
 

 \bibitem{MMNO} J. Mateu, P. Mattila, A. Nicolau and J. Orobitg, BMO for nondoubling measures, Duke Math. J. {\bf 102} (2000), no. 3, 533--565.
 
 \bibitem{NTV1}  F. Nazarov, S. Treil and A. Volberg,  Cauchy integral and Calder\'o--Zygmund operators on nonhomogeneous spaces, Int. Math. Res. Not. {\bf 15} (1997), 703--726.

 \bibitem{NTV2}  F. Nazarov, S. Treil and A. Volberg, Weak type estimates and Cotlar inequalities for Calder\'on-Zygmund operators on nonhomogeneous spaces. Int. Math. Res. Not. {\bf 9} (1998) 463--487.

 \bibitem{NTV3}  F. Nazarov, S. Treil and A. Volberg, The $Tb$-theorem on non-homogeneous spaces, Acta Math. {\bf 190} (2003), 151--239.

 \bibitem{OP} J. Orobitg and C. P\'erez, $A_p$ weights for nondoubling measures in $\mathbb{R}^n$ and applications, Trans. Amer. Math. Soc. {\bf 354} (2002), no. 5, 2013--2033.
 
 \bibitem{Sa} Y. Sawano, Generalized Morrey spaces for non--doubling measures, NoDEA Nonlinear Differential Equations Appl. {\bf 15} (2008), 413--425. 
 
 \bibitem{Stein} E. M. Stein, Harmonic analysis: real--variable methods, orthogonality, and oscillatory integrals. With the assistance of Timothy S. Murphy. Princeton Mathematical Series, 43. Monographs in Harmonic Analysis, III. Princeton University Press, Princeton, NJ, 1993. xiv+695 pp.
 
 \bibitem{To1} X. Tolsa,  BMO, $H^1$, and Calder\'on--Zygmund operators for non doubling measures. Math. Ann. {\bf 319} (2001), 89--149.

 \bibitem{To2}  X. Tolsa, A proof of the weak $(1,1)$ inequality for singular integrals with non doubling measures based on a Calder\'on--Zygmund decomposition, Publ. Mat. {\bf 45} (2001), 163--174.

\bibitem{To3} X. Tolsa, Littlewood--Paley theory and the $T(1)$ theorem with non--doubling measures, Advances in Mathematics {\bf 164}, 57--116.



\end{thebibliography}
\end{document}